\documentclass[10pt]{amsart}
\usepackage{amsmath}
\usepackage{amsxtra}
\usepackage{amscd}
\usepackage{amsthm}
\usepackage{amsfonts}
\usepackage{amssymb}
\usepackage{eucal}
\usepackage[all]{xy}
\usepackage{graphicx}
\usepackage[usenames]{color}
\usepackage{tikz-cd}
\usepackage{mathtools}
\usepackage{hyperref}
\usetikzlibrary{spath3}
\tikzset{between/.style n args={2}{/tikz/spath/at end path construction={
    \tikzset{spath/split at keep middle={current}{#1}{#2}}
}}}

\newtheorem{lem}[subsubsection]{Lemma}
\newtheorem{prop}[subsubsection]{Proposition}

\newtheorem{thm}[subsubsection]{Theorem}

\newtheorem{defn}[subsubsection]{Definition}

\theoremstyle{remark}
\newtheorem{rem}[subsubsection]{Remark}
\newtheorem{example}[subsubsection]{Example}

\newcommand{\on}[1]{\operatorname{#1}}

\newcommand{\Bun}{\on{Bun}}
\newcommand{\Dmod}{\on{D-mod}}
\newcommand{\Dmodh}{\on{D-mod}_{\frac{1}{2}}}
\newcommand{\QCoh}{\on{QCoh}}
\newcommand{\IndCoh}{\on{IndCoh}}

\newcommand{\gmod}{\on{\hat{\mathfrak{g}}-mod}}
\newcommand{\KL}{\on{KL}(G)_{\on{crit}}}
\newcommand{\IKL}{\on{IKL}(G)_{\on{crit}}}
\newcommand{\IKLx}{\on{IKL}(G)_{\on{crit}, x_0}}

\newcommand{\Gcheck}{\check{G}}
\newcommand{\gcheck}{\check{\mathfrak{g}}}
\newcommand{\ncheck}{\check{\mathfrak{n}}}
\newcommand{\Bcheck}{\check{B}}
\newcommand{\Tcheck}{\check{T}}

\newcommand{\Ran}{\on{Ran}}

\newcommand{\FactAlgCat}{\on{FactAlgCat}}
\newcommand{\FactModCat}{\on{FactModCat_{x_0}}}

\newcommand{\sssec}[1]{\subsubsection{}}

\title{Iwahori Fundamental Local Equivalence}
\author{Taeuk Nam}
\date{\today}

\begin{document}

\begin{abstract}
    We construct three tamely ramified local equivalences of factorization module categories. The first is a factorization version of the Arkhipov-Bezrukavnikov equivalence at a point. The second is a factorization version of the Bezrukavnikov equivalence at a point. The third is an Iwahori-ramified version of the factorizable Fundamental Local Equivalence.
\end{abstract}

\maketitle

\tableofcontents

%%%%%%%%%%%%%%%%%%%%%%%%%%%%%%%%%%%%%%%%%%%%%%%%%%%%%%%%%%%%%%%%%%%%%%%%%%

\section*{Introduction}
\subsection{Motivation}
\sssec{}
The (de Rham, global, critical level) unramified geometric Langlands conjecture was recently proven in a series of five papers \cite{GLC1}, \cite{GLC2}, \cite{GLC3}, \cite{GLC4}, and \cite{GLC5} by nine authors: Arinkin, Beraldo, Campbell, Chen, Faergeman, Gaitsgory, Lin, Raskin, and Rozenblyum. The statement is that there is an equivalence of DG categories
$$
\Dmod(\Bun_G) \overset{\simeq}{\longrightarrow} \IndCoh_{\on{Nilp}}(\on{LocSys}_{\Gcheck})
$$
satisfying some compatibilities.

\sssec{}
In the first paper, \cite{GLC1}, the authors construct the Langlands functor
$$
\mathbb{L}_G : \Dmod(\Bun_G) \longrightarrow \IndCoh_{\on{Nilp}}(\on{LocSys}_{\Gcheck}).
$$
The main ingredients of the construction are the \emph{spectral action}
$$
\QCoh(\on{LocSys}_{\Gcheck}) \otimes \Dmod(\Bun_G) \to \Dmod(\Bun_G)
$$
and the fact that 
$$
\Dmod(\Bun_G)
$$
is a compactly generated DG category.

\sssec{}
In the second paper, \cite{GLC2}, the Langlands functor is studied \emph{locally}. More precisely, we mean the following.

\sssec{}
At a point $x \in X$, we can consider three equivalences of a local nature. These are the geometric Casselman-Shalika equivalence
$$
\on{CS}_{G, x} : \on{Whit}^!(G)_{x} \overset{\simeq}{\longrightarrow} \on{Rep}(\Gcheck)_x,
$$
the geometric Satake equivalence
$$
\on{Sat}_{G, x} : \on{Sph}_{G, x} \overset{\simeq}{\longrightarrow} \on{Sph}_{\Gcheck, x}^{\on{spec}},
$$
and the Fundamental Local Equivalence (at critical level)
$$
\on{FLE}_{G, x} : \on{KL}(G)_{\on{crit}, x} \overset{\simeq}{\longrightarrow} \IndCoh(\on{Op}_{\Gcheck, x}^{\on{mon-free}}).
$$

Furthermore, the geometric Satake equivalence is an equivalence of \emph{monoidal categories}. There is an action of $\on{Sph}_{G, x}$ on $\on{Whit}^!(G)_{x}$ (resp. $\on{KL}(G)_{\on{crit}, x}$) and an action of $\on{Sph}_{\Gcheck, x}^{\on{spec}}$ on $\on{Rep}(\Gcheck)_x$ (resp. $\IndCoh(\on{Op}_{\Gcheck, x}^{\on{mon-free}})$). $\on{CS}_{G, x}$ (resp. $\on{FLE}_{G, x}$) is compatible with these actions under $\on{Sat}_{G, x}$.

\sssec{}
These equivalences are shown to be compatible with the Langlands functor in loc. cit. in the following sense: there are commutative diagrams
$$
\begin{tikzcd}
	{\on{Whit}^!(G)_{x}} && {\on{Rep}(\Gcheck)_{x}} \\
	\\
	{\Dmod(\Bun_G)} && {\IndCoh_{\on{Nilp}}(\on{LocSys}_{\Gcheck})} \\
	\\
	{\on{KL}(G)_{\on{crit}, x}} && {\IndCoh(\on{Op}_{\Gcheck, x}^{\on{mon-free}}).}
	\arrow["{\on{CS}_{G, x}}", from=1-1, to=1-3]
	\arrow[from=3-1, to=1-1]
	\arrow["{\mathbb{L}_G}", from=3-1, to=3-3]
	\arrow[from=3-3, to=1-3]
	\arrow[from=5-1, to=3-1]
	\arrow["{\on{FLE}_{G, x}}", from=5-1, to=5-3]
	\arrow[from=5-3, to=3-3]
\end{tikzcd}
$$

Furthermore, there is an action of $\on{Sph}_{G, x}$ on $\Dmod(\Bun_G)$ and an action of $\on{Sph}_{\Gcheck, x}^{\on{spec}}$ on $\IndCoh_{\on{Nilp}}(\on{LocSys}_{\Gcheck})$. $\mathbb{L}_G$ is compatible with these actions under $\on{Sat}_{G,x}$.

\sssec{}
However, this is insufficient to control the behaviour of the global Langlands functor. Instead of choosing a point $x \in X$, we must consider \emph{all} points of $X$ simultaneously. Not only that, we also want to allow these points to "move around in $X$", or in other words, we want to allow $x$ to vary in \emph{families}. In addition, we want to consider what happens when the points "collide with each other" while moving around in $X$, or in other words, we want to also allow families of \emph{tuples} of points of $X$.

\sssec{}
It turns out that the correct technical device to employ in order to encode the above information is \emph{Factorization}. The above equivalences are restrictions to a point $x \in \Ran$ of equivalences of factorization categories
$$
\on{CS}_{G} : \on{Whit}^!(G) \overset{\simeq}{\longrightarrow} \on{Rep}(\Gcheck),
$$
$$
\on{Sat}_{G} : \on{Sph}_{G} \overset{\simeq}{\longrightarrow} \on{Sph}_{\Gcheck}^{\on{spec}},
$$
and
$$
\on{FLE}_{G} : \on{KL}(G)_{\on{crit}} \overset{\simeq}{\longrightarrow} \IndCoh(\on{Op}_{\Gcheck}^{\on{mon-free}}).
$$

\sssec{}
If we consider tame (Iwahori) ramification at a fixed point $x_0 \in X$ in the context of global geometric Langlands, we are naturally led to ask: can we find a spectral description of the DG category of D-modules on $\Bun_G$ with Iwahori level structure? In other words, what should be on the right hand side of 
$$
\Dmod(\Bun_G^{\on{I}}) \simeq \ ?
$$

\sssec{}
We have strong heuristic reasons to expect that the answer should be Ind-coherent sheaves (with some singular support condition, that we will continue to call $\on{Nilp}$) on the stack
$$
\on{LocSys}_{\Gcheck}(X \setminus x_0) \underset{\on{LS}_{\Gcheck}(D_{x_0}^{\times})}{\times} \on{LS}_{\Bcheck}(D_{x_0}^{\times}) \underset{\on{LS}_{\Tcheck}(D_{x_0}^{\times})}{\times} \on{LS}_{\Tcheck}(D_{x_0}).
$$

\sssec{}
Using Faergeman's construction of a tamely ramified spectral action in \cite{Fae} and our proof of the compact generation of $\Dmod(\Bun_G^{\on{I}})$ in \cite{Nam}, we are able to emulate \cite{GLC1} to construct a functor
$$
\mathbb{L}_G^{\on{I}} : \Dmod(\Bun_G^{\on{I}}) \overset{\simeq}{\longrightarrow} \IndCoh_{\on{Nilp}}(\on{LocSys}_{\Gcheck}(X \setminus x_0) \underset{\on{LS}_{\Gcheck}(D_{x_0}^{\times})}{\times} \on{LS}_{\Bcheck}(D_{x_0}^{\times}) \underset{\on{LS}_{\Tcheck}(D_{x_0}^{\times})}{\times} \on{LS}_{\Tcheck}(D_{x_0})).
$$

\sssec{}
Following the strategy of the five papers, the next logical step is to study $\mathbb{L}_G^{\on{I}}$ locally. Since away from $x_0$ the functor is unramified, we should more specifically study it locally at $x_0$. The three pointwise equivalences in the unramified setting have natural Iwahori-ramified counterparts:
these are the Arkhipov-Bezrukavnikov equivalence
$$
\on{AB}_{x_0} : \on{Whit}^!(\on{Fl}_G)_{x_0} \overset{\simeq}{\longrightarrow} \QCoh(\ncheck/\Bcheck)_{x_0},
$$
Bezrukavnikov's equivalence
$$
\on{B}_{x_0} : \on{Aff}_{G, x_0} \overset{\simeq}{\longrightarrow} \on{Aff}_{\Gcheck, x_0}^{\on{spec}},
$$
and the Iwahori Fundamental Local Equivalence
$$
\on{IFLE}_{x_0} : \IKLx \overset{\simeq}{\longrightarrow} \IndCoh(\on{Op}_{\Gcheck, x_0}^{\on{mer}} \underset{\on{LS}_{\Gcheck}(D_{x_0}^{\times})}{\times} \on{LS}_{\Bcheck}(D_{x_0}^{\times}) \underset{\on{LS}_{\Tcheck}(D_{x_0}^{\times})}{\times} \on{LS}_{\Tcheck}(D_{x_0}).
$$

\sssec{}
We expect these pointwise equivalences to interact with $\mathbb{L}_G^{\on{I}}$ in an entirely analogous way to the unramified setting, but we again do not expect this to be sufficient. We want to allow $x \in X$ to vary in families, but because the Iwahori subgroup is not factorizable, we do not want to allow $x_0$ to move around. In fact, we want to allow families of tuples of points in $X$ that contain the constant family with value $x_0$ so we can consider what happens when the "mobile" point $x$ "collides" with the "stationary" point $x_0$.

\sssec{}
The technical device we propose to use in this case is the theory of \emph{Factorization Module Categories}. To be a bit more precise, we will consider \emph{pairs} $(\mathcal{C}, \mathcal{M})$ consisting of $\mathcal{C}$, a factorization category and $\mathcal{M}$, a factorization $\mathcal{C}$-module category at $x_0$.

\sssec{}
In this paper, we will show that each of the six local Iwahori-ramified categories at $x_0$ can be endowed with the structure of factorization module category over the factorization category corresponding to its unramified analogue. Furthermore, we will upgrade the three pointwise Iwahori-ramified equivalences to equivalences of pairs
$$
\on{AB} : (\on{Whit}^!(G), \on{Whit}^!(\on{Fl}_G)) \overset{\simeq}{\longrightarrow} (\on{Rep}(\Gcheck), \QCoh(\ncheck/\Bcheck)),
$$
$$
\on{B} : (\on{Sph}_G, \on{Aff}_{G}) \overset{\simeq}{\longrightarrow} (\on{Sph}_{\Gcheck}^{\on{spec}}, \on{Aff}_{\Gcheck}^{\on{spec}}),
$$
and
$$
\on{IFLE} : (\KL, \IKL) \overset{\simeq}{\longrightarrow} (\IndCoh^*(\on{Op}_{\Gcheck}^{\on{mon-free}}), \IndCoh^*(\on{Op}_{\Gcheck}^{\on{mer}} \underset{\on{LS}_{\Gcheck}^{\on{mer}}}{\times} \ncheck/\Bcheck)).
$$

\subsection{Main Results}
\sssec{}
Here are the main results in this paper.

\sssec{}
In Section \ref{Factorization}, we construct the six factorization module categories that are of interest to us. We will think of them and denote them as \emph{pairs} $(\mathcal{C}, \mathcal{M})$ consisting of a factorization category $\mathcal{C}$ and a factorization $\mathcal{C}$-module category (at $x_0$) $\mathcal{M}$. In all six cases, the factorization categories (i.e. the first components) are defined in \cite{GLC2}.

\begin{prop} \cite{GLC2}
    There is an equivalence, called the geometric Casselman-Shalika equivalence, of factorization categories
    $$
    \on{CS}_G : \on{Whit}^!(G) \longrightarrow \on{Rep}(\Gcheck)
    $$
    in $\FactAlgCat$.
\end{prop}

\sssec{}
In Section \ref{Section AB}, we construct a morphism of pairs
$$
\on{AB} : (\on{Whit}^!(G), \on{Whit}^!(\on{Fl}_G)) \longrightarrow (\on{Rep}(\Gcheck), \QCoh(\ncheck/\Bcheck))
$$
whose first component is $\on{CS}_G$. We also prove

\begin{thm} \label{Theorem AB}
    The morphism of pairs $\on{AB}$ is an equivalence in $\FactModCat$.
\end{thm}

\begin{prop} \cite{GLC2}
    There is a monoidal equivalence, called (derived) geometric Satake
    $$
    \on{Sat}_{G} : \on{Sph}_G \longrightarrow \on{Sph}_{\Gcheck}^{\on{spec}}
    $$
    in $\on{Alg}(\FactAlgCat)$. The equivalence $\on{CS}_G$ is equivariant with respect to $\on{Sat}_G$.
\end{prop}

\sssec{}
In Section \ref{Section B}, we construct a morphism of monoidal pairs
$$
\on{B} : (\on{Sph}_G, \on{Aff}_G) \longrightarrow (\on{Sph}_{\Gcheck}^{\on{spec}}, \on{Aff}_{\Gcheck}^{\on{spec}})
$$
whose first component is $\on{Sat}_G$. We also prove

\begin{thm} \label{Theorem B}
    The morphism of monoidal pairs $\on{B}$ is an equivalence in $\on{Alg}(\FactModCat)$, such that the equivalence of pairs $\on{AB}$ is equivariant with respect to it.
\end{thm}

\begin{prop} \cite{GLC2}
    There is an equivalence, called the Fundamental Local Equivalence, of factorization categories
    $$
    \on{FLE}: \on{KL}(G)_{\on{crit}} \longrightarrow \IndCoh^*(\on{Op}_{\Gcheck}^{\on{mon-free}})
    $$
    that is equivariant with respect to $\on{Sat}_G$.
\end{prop}

\sssec{}
In Section \ref{Section IFLE}, we construct a morphism of pairs
$$
\on{IFLE} : (\on{KL}(G)_{\on{crit}}, \on{IKL}(G)_{\on{crit}}) \longrightarrow (\IndCoh^*(\on{Op}_{\Gcheck}^{\on{mon-free}}), \IndCoh^*(\on{Op}_{\Gcheck}^{\on{mer}} \underset{\on{LS}_{\Gcheck}^{\on{mer}}}{\times} \ncheck/\Bcheck))
$$
whose first component is $\on{FLE}$. We also prove
\begin{thm} \label{Theorem IFLE}
    The morphism of pairs $\on{IFLE}$ is an equivalence in $\FactModCat$ that is equivariant with respect to $\on{B}$.
\end{thm}

\subsection{The Structure of the Paper}
\sssec{}
We will now outline the structure of this paper.

\sssec{}
We begin by recalling the basics of factorization in Subsection \ref{Factorization Basics}. We do this mainly to fix notation, and refer the reader to \cite{factorization} for details.

\sssec{}
In Subsections \ref{Main Characters of AB}, \ref{Main Characters of B}, and \ref{Main Characters of IFLE}, we define the six factorization module categories that appear on either side of $\on{AB}$, $\on{B}$, and $\on{IFLE}$ respectively.

\sssec{}
In Subsection \ref{Approximation}, we prove Theorem \ref{approximation}, which allows us to approximate various factorization module categories by categories of factorization modules. This will be used to construct $\on{B}$ and $\on{IFLE}$. This is a factorization module categorical version of Proposition \ref{alg approximation}, which, informally speaking, approximates factorization categories by categories of factorization modules, and was proven in \cite{GLC2}.

\sssec{}
The next three sections are each devoted to constructing the equivalences $\on{AB}$, $\on{B}$, and $\on{IFLE}$ respectively.

\sssec{}
In Subsection \ref{Fusability}, we discuss the notion of \emph{Fusability}, which will appear in an upcoming paper \cite{BCG} of Bogdanova, Chen, and Gaitsgory. Fusability is a condition on factorization $\on{Rep}(\Gcheck)$-module categories that can be checked combinatorially, and it allows us to reduce showing that a functor between factorization $\on{Rep}(\Gcheck)$-module categories is an equivalence to showing that the fiber at $x_0$ is an equivalence of plain categories. We argue that $\on{Whit}^!(\on{Fl}_G)$ is fusable.

\sssec{}
In Subsection \ref{Construction of AB}, we construct the morphism of pairs $\on{AB}$ using a semi-infinite version constructed in \cite{GLC3}. Loc. cit. also shows that semi-infinite version of $\on{AB}$ is a pointwise equivalence, which allows us to conclude that $\on{AB}_{x_0}$ is an equivalence in Subsection \ref{AB is a Pointwise Equivalence}, thereby proving Theorem \ref{Theorem AB}.

\sssec{}
In Subsection \ref{Conservativity}, we show that the 2-functor of factorization restriction, defined in the unpublished document \cite{CF} of Chen and Fu, along $\on{Rep}(\Gcheck) \longrightarrow \on{Sph}_{\Gcheck}^{\on{spec}}$ is conservative. Because the image of $\on{Aff}_G$ under this functor is fusable, this will allow us to reduce showing that $\on{B}$ is an equivalence to showing that $\on{B}_{x_0}$ is an equivalence.

\sssec{}
In Subsection \ref{Construction of B}, we construct the monoidal morphism of pairs $\on{B}$ using the action of $\on{Aff}_G$ on $\on{QCoh}(\ncheck/\Bcheck)$ via $\on{AB}$. To show that $\on{B}_{x_0}$ is an equivalence, in Subsection \ref{B is a Pointwise Equivalence} we adapt ideas of Dhillon and Taylor contained in \cite{DT} from the Betti setting. This proves Theorem \ref{Theorem B}.

\sssec{}
In Subsection \ref{Construction of IFLE}, we construct $\on{IFLE}$. In Subsection \ref{Reduction to the Pointwise Equivalence}, we argue that $\IKL$ is compactly generated and that $\on{IFLE}$ preserves compacts, in order to reduce showing that $\on{IFLE}$ is an equivalence to showing that $\on{IFLE}_{x_0}$ is an equivalence.

\sssec{}
In Subsection \ref{Affine Temperedness}, we introduce the notion of \emph{Iwahori-Hecke Temperedness}. This is a condition on $\on{Aff}_{G, x_0}$-module categories that implies that tensoring $\on{Whit}^!(\on{Fl}_G)_{x_0} \underset{\on{Aff}_{G, x_0}}{\otimes} -$ is conservative.

\sssec{}
Finally, in Subsection \ref{IFLE is a Pointwise Equivalence}, we use the compatibility between $\on{B}$ and $\on{IFLE}$ to embed $\on{Whit}^!(\on{Fl}_G)_{x_0} \underset{\on{Aff}_{G, x_0}}{\otimes} -$ applied to both sides of $\on{IFLE}$ in an ambient category $\on{IndCoh}(\on{Op}_{\Gcheck, x_0}^{\on{mer}})$. Then we show that these embeddings match, and that the two sides of $\on{IFLE}$ are Iwahori-Hecke tempered to conclude the proof of Theorem \ref{Theorem IFLE}.

\subsection{Acknowledgements}
The author would like to express their deepest gratitude towards their advisor Dennis Gaitsgory, for his invaluable mentorship and support throughout the course of this project.

The author is also greatful to German Stefanich, Ko Aoki, Kevin Lin, Wyatt Reeves, Ekaterina Bogdanova, Sam Raskin, Lin Chen, Justin Campbell, Gurbir Dhillon, Elden Elmanto, Natalie Stewart, Isabel Longbottom, Dhilan Lahoti, Eunice Sukarto, Sonja Farr, and countless other people for the numerous helpful conversations and patient explanations without which this paper would not exist.

Part of this project was done during the author’s time at the Max Planck Institute for Mathematics in Bonn.

\section{Factorization} \label{Factorization}
\subsection{Factorization Basics} \label{Factorization Basics}
In this subsection, we will briefly review some definitions and concepts related to factorization. The purpose of this is mainly to fix notation; for a detailed discussion, see \cite{factorization}.

\sssec{}
Throughout this paper, we will fix for all time a smooth, projective, connected curve $X$ over an algebraically closed field $k$ of characteristic zero. We will also fix a closed point $x_0 \in X$.

\sssec{}
We begin by fixing notation for the spaces on which our factorization objects will live. Informally speaking, the pattern is as follows: our "algebra-like" objects (factorization categories, factorization algebras, factorization spaces) will live on $\Ran$, and after fixing an "algebra-object", a "module-like" object of the corresponding type (factorization module categories, factorization modules, factorization module spaces) will live on $\Ran_{x_0}$. This is because for us, the intended application of factorization "module-like" objects are to study level structure at the fixed point $x_0$.

%Ran Space
\begin{defn}
    Denote by $\Ran$ the unital Ran (categorical) prestack of $X$. For a test affine scheme $S$,
    $$
    \Ran(S) = \{finite \ subsets \ of \ X_{dR}(S)\}
    $$
    which is considered as a partially ordered set under inclusion.
\end{defn}

\begin{rem}
    We are taking the convention that $\Ran$ without any decorations is the unital version of Ran space, while other sources may use the same notation to mean the non-unital version, whose outputs are the \emph{set} of finite subsets of $X_{dR}(S)$ rather than the poset (and possibly excluding the empty subset).
    \\
    From now on, everything related to factorization will refer to the unital version by default unless explicitly stated otherwise. One exception is that factorization functors (i.e. morphisms of factorization categories and factorization module categories) are allowed to be $\emph{lax}$-unital rather than strictly unital.
\end{rem}

\sssec{}
A variant of this is the prestack $\Ran^{\subset}$, which classifies ordered pairs of finite subsets of $X$ such that the first is contained in the second. There are projection maps
$$
\on{pr}_{\on{small}}, \on{pr}_{\on{big}}: \Ran^{\subset} \longrightarrow \Ran
$$
sending a pair $A \subset B \subset X_{dR}(S)$ to $A$ and $B$ respectively.

%Ran_x_0 space
\begin{defn}
    Let $Y \to \Ran$ be a morphism. Then denote
    $$
    \Ran_Y = Y \underset{\Ran}{\times} \Ran^{\subset}
    $$
    where $\on{pr}_{\on{small}}$ is used to form the fiber product. Then there is a map
    $$
    \on{pr}_{\on{big}, Y} : \Ran_Y \longrightarrow \Ran
    $$
    induced by $\on{pr}_{\on{big}} : \Ran^{\subset} \to \Ran$.
    \\
    The most important case for us is when $Y = \on{pt}$ and $Y \to \Ran$ is given by the singleton containing $x_0$. In this case, we denote the corresponding prestack by $\Ran_{x_0}$.
\end{defn}

\begin{rem}
    For a test affine scheme $S$, $\Ran_{x_0}(S)$ classifies finite subsets of $X_{dR}(S)$ containing the constant map $S \to \on{pt} \overset{x_0}{\to} X_{dR}$. Thus, $\on{pr}_{\on{big}, x_0} : \Ran_{x_0} \to \Ran$ is an inclusion.
\end{rem}

\sssec{}
Next, we will define the notion of factorization category. Intuitively, this is a way to assign a category to each tuple of points $\underline{x} \subset X$ in families, in a way that cares about the collisions of points. Consider the subset
$$
    [\Ran \times \Ran]_{disj} \subset \Ran \times \Ran
$$
classifying disjoint pairs of finite subsets of $X$. Then form the diagram
$$
\Ran \times \Ran \supset [\Ran \times \Ran]_{disj} \overset{\cup}{\to} \Ran.
$$
%FactAlgCat
\begin{defn}
    A factorization category is a sheaf of categories $\mathcal{C}$ on $\Ran$ together with factorization data, which consists of an equivalence
    $$
        (\mathcal{C} \boxtimes \mathcal{C})|_{disj} \simeq \cup^* \mathcal{C}
    $$
    as well as higher homotopy coherences.
    \\
    We will denote the category of factorization categories by $\FactAlgCat$.
\end{defn}

\sssec{}
To define the notion of factorization module category (at $x_0$), we will consider the subset
$$
    [\Ran \times \Ran_{x_0}]_{disj} \subset \Ran \times \Ran_{x_0}
$$
defined in the obvious way. We can again form a diagram
$$
\Ran \times \Ran_{x_0} \supset [\Ran \times \Ran_{x_0}]_{disj} \overset{\cup}{\to} \Ran_{x_0}.
$$

%FactModCat_x_0
\begin{defn}
    Let $\mathcal{C}$ be a factorization category. A factorization $\mathcal{C}$-module category (at $x_0$) is a sheaf of categories $\mathcal{M}$ on $\Ran_{x_0}$ together with factorization module data over $\mathcal{C}$, which consists of an equivalence 
    $$
        (\mathcal{C} \boxtimes \mathcal{M})|_{disj} \simeq \cup^* \mathcal{M}
    $$
    as well as higher homotopy coherences.
    \\
    We will denote the category of factorization $\mathcal{C}$-module categories by $\mathcal{C}-\FactModCat$.
\end{defn}

\begin{rem}
    Instead of fixing a factorization category $\mathcal{C}$, we sometimes wish to consider \emph{pairs} $(\mathcal{C}, \mathcal{M})$, where $\mathcal{M}$ is a factorization $\mathcal{C}$-module category. The category of such pairs is denoted $\FactModCat$, and a morphism
    $$
    (\mathcal{C}_1, \mathcal{M}_1) \to (\mathcal{C}_2, \mathcal{M}_2)
    $$
    in this category consists of a factorization functor (i.e. a morphism in $\FactAlgCat$)
    $$
    \mathcal{C}_1 \to \mathcal{C}_2
    $$
    and a morphism of sheaf of categories (on $\Ran_{x_0}$)
    $$
    \mathcal{M}_1 \to \mathcal{M}_2
    $$
    compatible with the above factorization functor.
\end{rem}

\sssec{}
We will now discuss factorization structures on the level of objects.

%FactAlg(A)
\begin{defn}
    Let $\mathcal{C}$ be a factorization category. A factorization algebra $a$ in $\mathcal{C}$ is a global section of $\mathcal{C}$ equipped with factorization data, which consists of an isomorphism
    $$
    (a \boxtimes a)|_{disj} \simeq \cup^!a
    $$
    as well as higher homotopy coherences.
    \\
    We denote the category of factorization algebras in $\mathcal{C}$ by $\on{FactAlg(\mathcal{C})}$.
\end{defn}

%a-FactMod(M)_x_0
\begin{defn}
    Let $Y \to \Ran$. Furthermore, let $(\mathcal{C}, \mathcal{M}) \in \on{FactModCat}_{Y}$ (whose definition is obtained by replacing $\Ran_{x_0}$ by $\Ran_{Y}$ in the definition of $\FactModCat$) be a pair. Then for a factorization algebra $a$ in $\mathcal{C}$ and any $Z \to \Ran_Y$, a factorization $a$-module $m$ in $\mathcal{M}$ at $Z$ is a section of $\mathcal{M}$ on $Z$ equipped with factorization module data over $a$, which consists of an isomorphism
    $$
    (a \boxtimes m)|_{disj} \simeq \cup^!m
    $$
    as well as higher homotopy coherences.
    \\
    Denote the category of such objects by $a\on{-FactMod}_{Z}(\mathcal{M})$.
\end{defn}

\sssec{}
The reason we define factorization modules for more than just $x_0$ is that we want to construct factorization categories and factorization module categories of factorization modules by letting $Z$ vary.

\begin{defn}
    Let $(\mathcal{C}, \mathcal{M}) \in \FactModCat$ and let $a \in \on{FactAlg}(\mathcal{C})$. Consider the sheaf of categories 
    $$
    a\on{-FactMod}(\mathcal{C})
    $$
    on $\Ran$ that assigns to each test affine scheme $S$ over $\Ran$ the category $a\on{-FactMod}_S(\mathcal{C})$. This has the natural structure of a factorization category. Furthermore, consider the sheaf of categories 
    $$
    a\on{-FactMod}(\mathcal{M})
    $$
    on $\Ran_{x_0}$ that assigns to each test affine scheme $S$ over $\Ran_{x_0}$ the category
    $a\on{-FactMod}_S(\mathcal{M})$. This has the natural structure of a factorization $a\on{-FactMod}(\mathcal{C})$-module category at $x_0$.
\end{defn}

\sssec{}
We have described an algebraic method for producing factorization categories and factorization module categories. We will finish this subsection by describing a geometric way to produce factorization categories and factorization module categories.

\begin{defn}
    A factorization space is a space $\mathcal{Y} \to \Ran$ equipped with factorization data, which includes an isomorphism
    $$
    (\mathcal{Y} \times \mathcal{Y}) \underset{\Ran \times \Ran}{\times} [\Ran \times \Ran]_{disj} \simeq \mathcal{Y} \underset{\Ran, \cup}{\times} [\Ran \times \Ran]_{disj}.
    $$
\end{defn}

\begin{defn}
    A factorization $\mathcal{Y}$-module space (at $x_0$) is a space $\mathcal{Z} \to \Ran_{x_0}$ equipped with factorization module data over $\mathcal{Y}$, which includes an isomorphism
    $$
    (\mathcal{Y} \times \mathcal{Z}) \underset{\Ran \times \Ran_{x_0}}{\times} [\Ran \times \Ran_{x_0}]_{disj} \simeq \mathcal{Z} \underset{\Ran_{x_0}, \cup}{\times} [\Ran \times \Ran_{x_0}]_{disj}.
    $$
\end{defn}

\sssec{}
    Given a "nice" sheaf theory with pullbacks and Kunneth formula, for example $\QCoh$, a factorization space $\mathcal{Y}$ gives rise to a factorization category via
    $$
    (S \to \Ran) \mapsto \QCoh(S \times_{\Ran} \mathcal{Y})
    $$
    and a factorization $\mathcal{Y}$-module space $\mathcal{Z}$ gives rise to a factorization $\QCoh(\mathcal{Y})$-module category via
    $$
    (S \to \Ran_{x_0}) \mapsto \QCoh(S \times_{\Ran_{x_0}} \mathcal{Z}).
    $$

\subsection{Main Characters of AB} \label{Main Characters of AB}
In this subsection and the two that follow, we will recall or define the various objects of $\FactAlgCat$ and $\FactModCat$ that will appear throughout this paper.

%Whit(Fl)
\sssec{}
We begin by recalling the definition of the factorization category $\on{Whit}^!(G)$ given in Section 1.3 of \cite{GLC2}. In order to do so, we first consider the Ran Grassmannian $\on{Gr}_{G, \on{Ran}}$. For a test affine scheme $S$ and $\underline{x} : S \to \on{Ran}$, a lift of $\underline{x}$ to $\on{Gr}_{G, \on{Ran}}$ is given by a pair
$$
(\mathcal{P}_G, \alpha)
$$
of a principal $G$-bundle $\mathcal{P}_G$ on $S \times X$, and a trivialization
$$
\alpha: \mathcal{P}_G|_{(S \times X) \setminus \Gamma_{\underline{x}}} \simeq \mathcal{P}_G^{triv}
$$
of $\mathcal{P}_G$ on the complement of the graph of $\underline{x}$ in $S \times X$.

\sssec{}
The Ran Grassmannian naturally has the structure of a (unital) factorization space, so we can produce a (unital) factorization category $\Dmod_{\frac{1}{2}}(\on{Gr}_G)$ (for a discussion of critically twisted D-modules, see Section 1.1 of \cite{GLC2}) defined by
$$
    \Dmod_{\frac{1}{2}}(\on{Gr}_G)_S = \Dmod_{\frac{1}{2}}(S \times_{\on{Ran}} \on{Gr}_{G, \on{Ran}}).
$$
Furthermore, there is an action of the factorization group ind-scheme $LG_{\on{Ran}}$ (and thus of the factorization subgroup ind-scheme $LN_{\on{Ran}}$) on $\on{Gr}_{G, \on{Ran}}$, and basically what we would like to do now is define $\on{Whit}^!(G)$ as the Whittaker invariant category of $\Dmod_{\frac{1}{2}}(\on{Gr}_G)$, i.e. as
$$
\Dmod_{\frac{1}{2}}(\on{Gr}_G)^{(LN, \chi)}
$$
or more precisely,
$$
S \mapsto \Dmod_{\frac{1}{2}}(\on{Gr}_G)_S^{(LN_S, \chi_S)}.
$$

\sssec{}
However, we must incorporate the geometric twisting construction from Section 1.2 of \cite{GLC2}, which we briefly repeat now.

Let $H \subset G$ is a subgroup, and let $Y$ be a space living over $X$ equipped with an $L^+H$-action. Also, let $\mathcal{P}_H$ be a principal $H$-bundle on $X$. Then the twist of $Y$ by $\mathcal{P}_H$ is
$$
    Y_{\mathcal{P}_H} = Y \times^{L^+H} (\mathcal{P}_H|_{D^\times}) = (Y \times (\mathcal{P}_H|_{D^\times}))/L^+H.
$$

In our case, we apply (a factorizable version of) the above construction to $\on{Gr_{G, \on{Ran}}}$, taking $H = T$ and $\mathcal{P}_T$ to be the factorization $T$-principal bundle $\rho(\omega_X) = 2\rho(\omega_X^{\otimes \frac{1}{2}})$.

Since $LN$ is normalized by $L^+T$, we can take the twist $LN_{\rho(\omega_X)}$. This is a (factorization) group ind-scheme which acts on $\on{Gr_{G, \on{Ran}, \rho(\omega_X)}}$. Finally, we define

\begin{defn}
    The (unital) factorization category $\on{Whit}^!(G) \in \FactAlgCat$ is
    $$
    \Dmod_{\frac{1}{2}}(\on{Gr}_{G, \rho(\omega_X)})^{(LN_{\rho(\omega_X)}, \chi_{\rho(\omega_X)})} .
    $$
\end{defn}

\sssec{}
Next, we will define the factorization $\on{Whit}^!(G)$-module category (at the point $x_0$) $\on{Whit}^!(\on{Fl}_G)$. Let $\on{Fl}_{G, \on{Ran}_{x_0}}$ be the following. For a test affine scheme $S$ and $\underline{x} : S \to \on{Ran}_{x_0}$, a lift of $\underline{x}$ to $\on{Fl}_{G, \on{Ran}_{x_0}}$ is given by a quadruple
$$
    (\mathcal{P}_G, \alpha, \mathcal{P}_B, \gamma)
$$
of a principal $G$-bundle $\mathcal{P}_G$ on $S \times X$, a trivialization
$$
\alpha: \mathcal{P}_G|_{(S \times X) \setminus \Gamma_{\underline{x}}} \simeq \mathcal{P}_G^{triv}
$$
of $\mathcal{P}_G$ on the complement of the graph of $\underline{x}$ in $S \times X$, and a reduction of structure group to $B$
$$
\gamma : \mathcal{P}_B \times^B G \simeq \mathcal{P}_G|_{S \times \{x_0\}}
$$
on $S \times \{x_0\}$.

\sssec{}
$\on{Fl}_{G, \on{Ran}_{x_0}}$ naturally has the structure of a factorization module space over the Ran Grassmannian. Furthermore, there is an action of the factorization module group ind-scheme $LG_{\on{Ran_{x_0}}}$ (i.e. the restriction of $LG_{\on{Ran}}$ to $\on{Ran_{x_0}} \subset \on{Ran}$ considered as a group object in factorization module spaces over $LG_{\on{Ran}}$) on $\on{Fl}_{G, \on{Ran}_{x_0}}$. Thus we can define $\on{Whit}^!(\on{Fl}_G)$ in the same way as $\on{Whit}^!(G)$ \emph{mutatis mutandis}, by forming the twist
$$
\on{Fl}_{G, \on{Ran}_{x_0}, \rho(\omega_X)}
$$
and taking twisted Whittaker invariants
$$
\Dmod_{\frac{1}{2}}(\on{Fl}_{G, \on{Ran}_{x_0}, \rho(\omega_X)})^{(LN_{\rho(\omega_X)}, \chi_{\rho(\omega_X)})}.
$$

\begin{defn}
    The pair $(\on{Whit}^!(G), \on{Whit}^!(\on{Fl}_G)) \in \FactModCat$ consists of the factorization category $\on{Whit}^!(G)$ and the factorization $\on{Whit}^!(G)$-module category 
    $$
    \on{Whit}^!(\on{Fl}_G) = \Dmod_{\frac{1}{2}}(\on{Fl}_{G, \on{Ran}_{x_0}, \rho(\omega_X)})^{(LN_{\rho(\omega_X)}, \chi_{\rho(\omega_X)})} .
    $$
\end{defn}

\sssec{}
Now we move on to the spectral side. In order to define the factorization category $\on{Rep}(\check{G})$ and the factorization $\on{Rep}(\check{G})$-module category $\QCoh(\ncheck/\Bcheck)$, we will recall the definition of regular and meromorphic local systems. A more thorough version of the discussion that will take up the remainder of this subsection can be found in Appendices B.3-B.7 of \cite{GLC2}.

\sssec{}
For an algebraic group $H$, the factorization space of regular $H$-local systems on the formal disk $\on{LS}_H^{\on{reg}}$ is defined as follows.

Given a test affine scheme $S$ and an $S$-point $\underline{x}: S \to \on{Ran}$, the formal disk around $\underline{x}$ is defined as the $X$-scheme
$$
    \hat{D}_{\underline{x}} = (S \times X) \underset{S_{dR} \times X_{dR}}{\times} \Gamma_{\underline{x}, dR}
$$
this descends to an $X_{dR}$-scheme since we can define
$$
     \hat{D}_{\underline{x}, \nabla} = (S \times X_{dR}) \underset{S_{dR} \times X_{dR}}{\times} \Gamma_{\underline{x}, dR} = S \underset{S_{dR}}{\times} \Gamma_{\underline{x}, dR}.
$$
and
$$
    \hat{D}_{\underline{x}} = \hat{D}_{\underline{x}, \nabla} \times_{X_{dR}} X.
$$

\sssec{}
Let $Y_{\nabla}$ be a space over $X_{dR}$, and denote
$$
Y = Y_{\nabla} \times_{X_{dR}} X.
$$
Now we will define the factorization space called the \emph{horizontal arc space of} $Y$ that is denoted by $L_{\nabla}^+Y$. A lift of $\underline{x}$ to $L_{\nabla}^+Y$ is given by an $X_{dR}$-map
$$
\hat{D}_{\underline{x}, \nabla} \to Y_{\nabla}.
$$

\sssec{}
We now define 
$$
\on{LS}_H^{\on{reg}} = L_{\nabla}^+(\on{pt}/H \times X).
$$
In other words, a lift of $\underline{x}$ to $\on{LS}_H^{\on{reg}}$ is given by the data of an $X_{dR}$-map
$$
\hat{D}_{\underline{x}, \nabla} \to \on{pt}/H \times X_{dR}.
$$
or equivalently a map
$$
\hat{D}_{\underline{x}, \nabla} \to \on{pt}/H.
$$

\begin{defn}
    The factorization category $\on{Rep}(\check{G}) \in \FactAlgCat$ is defined as 
    $$
    \QCoh(\on{LS}_{\check{G}}^{\on{reg}}).
    $$
\end{defn}

%\begin{rem}
%A priori, $\IndCoh$ is defined only on locally almost of finite type prestacks; $\IndCoh^*$ (and $\IndCoh^!$) extends the definition to infinite-type prestacks via Kan extension using *-push (resp. !-pull). See Appendix A.5 (resp. Appendix A.4) of \cite{GLC2} for the definition.
%\end{rem}

\sssec{}
The factorization space of meromorphic $H$-local systems on the formal disk $\on{LS}_{H}^{\on{mer}}$ is more subtle to define. For a test affine scheme $S$ and an $S$-point $\underline{x} : S \to \Ran$, we note that $\hat{D}_{\underline{x}}$ is an ind-affine ind-scheme. Writing
$$
\hat{D}_{\underline{x}} = \on{colim} \on{Spec}(A_i),
$$
we denote
$$
D_{\underline{x}} = \on{Spec}(\on{lim} A_i),
$$
or, in other words, we take the colimit in the category of affine schemes instead. The \emph{punctured formal disk around} $\underline{x}$ is then defined to be
$$
D_{\underline{x}}^{\times} = D_{\underline{x}} \setminus \Gamma_{\underline{x}}.
$$

\begin{example}
    If we take $S = \on{Spec} (k)$ and $\underline{x}$ to be the singleton set containing only $x \in X$, and we let $t$ be a local coordinate of $X$ at $x$, then
    $$
    \hat{D}_{\underline{x}} = \on{colim} \on{Spec}(k[t]/t^i) = \on{Spf} (k[[t]])
    $$
    and
    $$
    D_{\underline{x}} = \on{Spec}(\on{lim} k[t]/t^i) = \on{Spec} (k[[t]]).
    $$
    In this case,
    $$
    D_{\underline{x}}^{\times} = \on{Spec} (k((t))).
    $$
\end{example}

\sssec{}
    According to Appendix A.1 in \cite{Katia}, the $X$-schemes $D_{\underline{x}}$ and $D_{\underline{x}}^{\times}$ both descend over $X_{dR}$. We denote the descents by
    $$
    D_{\underline{x}, \nabla}
    $$
    and
    $$
    D_{\underline{x}, \nabla}^{\times}
    $$
    respectively. If $Y_{\nabla} \to X_{dR}$ is affine, an $X_{dR}$-map
    $$
    \hat{D}_{\underline{x}, \nabla} \to Y_{\nabla}
    $$
    is equivalent to an $X_{dR}$-map
    $$
    D_{\underline{x}, \nabla} \to Y_{\nabla}
    $$
    so we can define the \emph{horizontal loop space of} $Y$, $L_{\nabla}Y$ as the factorization space such that a lift of $\underline{x}$ to $L_{\nabla}Y$ is given by an $X_{dR}$-map
    $$
    D_{\underline{x}, \nabla}^{\times} \to Y_{\nabla}.
    $$

\begin{rem}
We emphasize that we are only defining $L_{\nabla}Y$ for \emph{affine} $Y_{\nabla} \to X_{dR}$. This is because we want there to be a map
$$
L_{\nabla}^+Y \to L_{\nabla}Y
$$
induced by
$$
D_{\underline{x}, \nabla}^{\times} \to D_{\underline{x}, \nabla}.
$$
Because of this, we cannot simply define $\on{LS}_{H}^{\on{mer}} = L_{\nabla}(\on{pt}/H \times X)$, since $\on{pt}/H \times X_{dR}$ is not affine over $X_{dR}$.
\end{rem}

\sssec{}
Next we will introduce the \emph{Jets construction}, which is described in Appendix B.5 of \cite{GLC2}. For a prestack $Y \to X$, we can produce an $X_{dR}$-prestack
$$
\on{Jets}(Y) \to X_{dR}
$$
whose $S$-points are pairs consisting of $x: S \to X_{dR}$ and an $X$-map
$$
\hat{D}_{x} \to Y.
$$
This gives us a way to express $L^+Y$, the factorization arc space of $Y$, as
$$
L^+Y = L_{\nabla}^+\on{Jets}(Y)
$$
Furthermore, if $Y \to X$ is affine, then $\on{Jets}(Y) \to X_{dR}$ is also affine, and we can express $LY$, the factorization loop space of $Y$, as
$$
LY = L_{\nabla} \on{Jets}(Y).
$$

\sssec{}
Now we are ready to define $\on{LS}_{H}^{\on{mer}}$. To motivate the definition, we note that as group objects over $X_{dR}$,
$$
H \times X_{dR} \subset \on{Jets}(H \times X)
$$
and the quotient
$$
\on{Jets}(H \times X)/(H \times X_{dR}) = \on{Jets}(\mathfrak{h} \otimes \omega_{X})
$$
the latter which we will denote by $\on{Conn}(\mathfrak{h})$. Note that this is affine over $X_{dR}$ since $\mathfrak{h} \otimes \omega_X$ is affine (in fact, a vector bundle) over $X$. Then
$$
\on{pt}/H \times X_{dR} = (\on{pt} \times X_{dR})/(H \times X_{dR}) = \on{Conn}(\mathfrak{h})/\on{Jets}(H \times X)
$$
and thus
$$
\on{LS}_{H}^{\on{reg}} = L_{\nabla}^+(\on{pt}/H \times X) = L_{\nabla}^+\on{Jets}(\mathfrak{h} \otimes \omega_{X})/L_{\nabla}^+\on{Jets}(H \times X) = L_{\nabla}^+\on{Conn}(\mathfrak{h})/L^+H.
$$
Defining
$$
\on{LS}_{H}^{\on{mer}} = L_{\nabla}\on{Conn}(\mathfrak{h})/LH
$$
gives a map
$$
\on{LS}_H^{\on{reg}} \to \on{LS}_{H}^{\on{mer}}
$$
as desired.

\sssec{}
We will finish this subsection by giving the definition of the factorization $\on{Rep}(\check{G})$-module category $\on{QCoh}(\ncheck/\Bcheck)$. We could of course form the factorization space
$$
\on{LS}_{\Bcheck}^{\on{mer}} \underset{\on{LS}_{\Tcheck}^{\on{mer}}}{\times} \on{LS}_{\Tcheck}^{\on{reg}},
$$
but we want a version of this as a factorization module space over $\on{LS}_{\Gcheck}^{\on{reg}}$ at $x_0$. Denote by $\on{LS}_H(D_{\Ran_{x_0}}^{\times})$ the factorization space (defined in Definition 3.1.8 of \cite{Katia}) such that the lifts of $\underline{x}: S \to \on{Ran}_{x_0}$ to $\on{LS}_H(D_{\Ran_{x_0}}^{\times})$ are given by
$$
\{D_{\underline{x}, \nabla} \underset{X_{dR}}{\times} (X \setminus x_0)_{dR} \to \on{Conn}(\mathfrak{h})\}/\{D_{\underline{x}, \nabla} \underset{X_{dR}}{\times} (X \setminus x_0)_{dR} \to H\}.
$$

Since $D_{S \times x_0, \nabla} \subset D_{\underline{x}, \nabla}$, there is a map via restriction
$$
\on{LS}_H(D_{\Ran_{x_0}}^{\times}) \to \on{LS}_H(D_{x_0}^{\times})
$$
where the latter is the fiber of $\on{LS}_H^{\on{mer}}$ at $\{x_0\} \in \Ran_{x_0}$.

We define ${\ncheck/\Bcheck}_{\Ran_{x_0}}$ to be the factorization $\on{LS}_{\Gcheck}^{\on{reg}}$-module space
$$
\on{LS}_{\Gcheck}(D_{\Ran_{x_0}}^{\times}) \underset{\on{LS}_{\Gcheck}(D_{x_0}^{\times})}{\times} \on{LS}_{\Bcheck}(D_{x_0}^{\times}) \underset{\on{LS}_{\Tcheck}(D_{x_0}^{\times})}{\times} \on{LS}_{\Tcheck}(D_{x_0}).
$$

\begin{defn}
    The pair $(\on{Rep}(\Gcheck), \QCoh(\ncheck/\Bcheck)) \in \FactModCat$ consists of the factorization category $\on{Rep}(\Gcheck)$ and the factorization $\on{Rep}(\Gcheck)$-module category
    $$
        \QCoh(\ncheck/\Bcheck) = \QCoh(\ncheck/\Bcheck_{\Ran_{x_0}}).
    $$
\end{defn}

\begin{rem}
    The fiber of $\ncheck/\Bcheck_{\Ran_{x_0}}$ at $\{x_0\} \in \Ran_{x_0}$ is
    $$
        \on{LS}_{\Bcheck}(D^{\times}) \underset{\on{LS}_{\Tcheck}(D^{\times})}{\times} \on{LS}_{\Tcheck}(D)
    $$
    which identifies with the smooth, finite dimensional stack
    $$
    \tilde{\check{\mathcal{N}}}/\check{G} = \ncheck/\Bcheck
    $$
    and the fiber of the sheaf of categories $\IndCoh^*(\ncheck/\Bcheck_{\Ran_{x_0}})$ is thus the usual category
    $$
    \QCoh(\ncheck/\Bcheck)
    $$
    which is why we abusively denote the factorization module category using the same name.
\end{rem}

\subsection{Main Characters of B} \label{Main Characters of B}

\sssec{}
We will now recall the definition from Section 1.5 of \cite{GLC2} of the spherical Hecke category $\on{Sph}_G$ as a monoidal factorization category. We consider the factorization category
$$
\Dmodh(\on{Gr}_{G})
$$
and take $L^+G$-invariants
$$
\Dmodh(\on{Gr}_{G})^{L^+G}
$$
which can also be written as
$$
\Dmodh(L^+G \backslash LG / L^+G)
$$
This factorization category has a monoidal structure given by convolution. 

\sssec{}
We then perform a renormalization: consider all objects in $\Dmodh(L^+G \backslash LG / L^+G)$ whose $!$-pullback along
$$
LG / L^+G \to L^+G \backslash LG/L^+G
$$
is compact (equivalently, such that the $!$-pullback along $L^+G \backslash LG \to L^+G \backslash LG/L^+G$ is compact). Then

\begin{defn}
    The Spherical Hecke category $\on{Sph}_G \in \FactAlgCat$ is the Ind-completion of the category of such objects. It inherits a monoidal structure from the non-renormalized version.
\end{defn}

\begin{rem}
    If we want $\on{Sph}_G$ to act on $\on{Whit}^!(G)$, then we should instead take $\on{Sph}_{G, \rho(\omega_X)}$, where we treat $\rho(\omega_X)$ as a principal $G$-bundle via $T \to G$, and use the geometric twisting construction from Section 1.2 of \cite{GLC2} to form
    $$
    \Dmodh(\on{Gr}_{G, \rho(\omega_X)})^{L^+G_{\rho(\omega_X)}, ren}.
    $$
    However, according to Section 1.5.8 of \cite{GLC2}, we have a canonical identification
    $$
    \on{Sph}_G \simeq \on{Sph}_{G, \rho(\omega_X)}
    $$
    because of the general fact about the geometric twisting construction that
    $$
    Y/L^+H \simeq Y_{\mathcal{P}_H}/L^+H_{\mathcal{P}_H}.
    $$
\end{rem}

%%%Define%%%Aff_G%%%
\sssec{}
Next, we will give the definition of the affine Hecke category $\on{Aff}_G$ as a monoidal factorization $\on{Sph}_G$-module category. More precisely, the pair $(\on{Sph}_G, \on{Aff}_G)$ will have the structure of monoidal object in $\FactModCat$.

\sssec{}
We start by defining a factorization module space version of the Iwahori subgroup $I \subset L^+G$. Recall that $L^+G_{\Ran}$ has the structure of a factorization group: for a test affine scheme $S$ and an $S$-point $\underline{x}: S \to \Ran$, a lift of $\underline{x}$ to $L^+G_{\Ran}$ is given by a map
$$
\hat{D}_{\underline{x}} \to G 
$$
or equivalently (because $G$ is affine)
$$
D_{\underline{x}} \to G.
$$
Then the Iwahori $I_{\Ran_{x_0}}$ is the factorization $L^+G_{\Ran}$-module group defined as follows. For a test affine scheme $S$ and an $S$-point $\underline{x}: S \to \Ran_{x_0}$, a lift of $\underline{x}$ to $I_{\on{Ran}_{x_0}}$ is given by a map
$$
D_{\underline{x}} \to G
$$
such that $S \times \{x_0\} \subset D_{\underline{x}}$ lands in $B \subset G$.

\sssec{}
Then we take the factorization $\Dmodh(\on{Gr}_G)$-module category
$$
\Dmodh(\on{Fl}_G)
$$
and take $I$-invariants
$$
\Dmodh(\on{Fl}_G)^{I} = \Dmodh(I\backslash LG/I)
$$
which is a factorization $\Dmodh(\on{Gr}_G)^{L^+G}$-module category with monoidal structure given by convolution. Renormalizing in the same fashion as before, i.e. taking the subcategory of objects whose $!$-pullback along
$$
LG/I \to I\backslash LG/I
$$
is compact (equivalently, such that the $!$-pullback along $I\backslash LG \to I\backslash LG/I$ is compact) and Ind-completing, we obtain

\begin{defn}
    The pair $(\on{Sph}_G, \on{Aff}_G) \in \FactModCat$ consists of the factorization category $\on{Sph}_G$ and the factorization $\on{Sph}_G$-module category $\on{Aff}_G$ produced just above. The pair $(\on{Sph}_G, \on{Aff}_G)$ inherits a monoidal structure from the non-renormalized (in both components) version.
\end{defn}

\begin{rem}
Just as in the spherical situation, in order for $\on{Aff}_G$ to act on $\on{Whit}^!(\on{Fl}_G)$ (more precisely, for the monoidal pair $(\on{Sph}_G, \on{Aff}_G)$ to act on the pair $(\on{Whit}^!(G), \on{Whit}^!(\on{Fl}_G))$), we should have used the geometric twisting construction by $\rho(\omega_X)$. However, just as in the spherical situation, we have a canonical identification between the twisted and untwisted versions.
\end{rem}

\sssec{}
Again, we will move on to the spectral side. The spectral spherical Hecke category $\on{Sph}_{\Gcheck}^{\on{spec}}$ is defined in Appendix E.2 (and discussed at length in the rest of Appendix E) of \cite{GLC2} as follows.

\sssec{}
Consider the spectral spherical Hecke stack
$$
\on{Hecke}_{\Gcheck}^{\on{spec}} = \on{LS}_{\Gcheck}^{\on{reg}} \underset{\on{LS}_{\Gcheck}^{\on{mer}}}{\times} \on{LS}_{\Gcheck}^{\on{reg}}
$$
which is a factorization space because $\on{LS}_{\Gcheck}^{\on{reg}}$ and $\on{LS}_{\Gcheck}^{\on{mer}}$ are. Then

\begin{defn}
    The factorization category $\on{Sph}_{\Gcheck}^{\on{spec}} \in \FactAlgCat$ is
    $$
    \IndCoh^*(\on{LS}_{\Gcheck}^{\on{reg}} \underset{\on{LS}_{\Gcheck}^{\on{mer}}}{\times} \on{LS}_{\Gcheck}^{\on{reg}}).
    $$
    However, we emphasize that this is \emph{not} just the $\IndCoh^*$ construction from Appendix A.5 of \cite{GLC2} applied to $\on{Hecke}_{\Gcheck}^{\on{spec}}$. Indeed, $\IndCoh^*$ is not even defined on $\on{Hecke}_{\Gcheck}^{\on{spec}}$. Instead, there is an identification
    $$
    \on{Hecke}_{\Gcheck}^{\on{spec}} \simeq L_{\nabla}^+G \backslash L_{\nabla}G/L_{\nabla}^+G
    $$
    and the definition of the notation
    $$
    \IndCoh^*(\on{LS}_{\Gcheck}^{\on{reg}} \underset{\on{LS}_{\Gcheck}^{\on{mer}}}{\times} \on{LS}_{\Gcheck}^{\on{reg}})
    $$
    is 
    $$
    \IndCoh^*(L_{\nabla}G)_{L_{\nabla}^+G \times L_{\nabla}^+G}.
    $$
    
    (Note that $L_{\nabla}G$ is an ind-affine ind-scheme, so $\IndCoh^*$ is defined on it. This is analogous to how $\QCoh_{co}$ is defined on ind-affine ind-schemes - see Appendix A.2 of \cite{GLC2}.)
    
    It has a monoidal structure given by convolution, described in Appendix E.6 of \cite{GLC2}. It acts on $\on{Rep}(\Gcheck)$ as described in Appendix E.7 of the same source.
\end{defn}

%\begin{rem}
%    A priori, $\IndCoh$ is defined only on locally almost of finite type prestacks; $\IndCoh^*$ (and $\IndCoh^!$) extends the definition to infinite-type prestacks via Kan extension using *-push (resp. !-pull). See Appendix A.5 (resp. Appendix A.4) of \cite{GLC2} for the definition.
%\end{rem}

\sssec{}
%Finally, we will give the definition of the factorization $\on{Sph}_{\Gcheck}^{\on{spec}}$-module category $\on{Aff}_{\Gcheck}^{\on{spec}}$.
Note that the pullback
$$
(\on{LS}_{\Gcheck}^{\on{mer}})_{\Ran_{x_0}} = \on{LS}_{\Gcheck}^{\on{mer}} \underset{\Ran}{\times} \Ran_{x_0}
$$
is a factorization $\on{LS}_{\Gcheck}^{\on{mer}}$-module space. Since $\ncheck/\Bcheck_{\Ran_{x_0}}$ is a factorization $\on{LS}_{\Gcheck}^{\on{reg}}$-module space, it follows that 
$$
\ncheck/\Bcheck_{\Ran_{x_0}} \underset{(\on{LS}_{\Gcheck}^{\on{mer}})_{\Ran_{x_0}}}{\times} \ncheck/\Bcheck_{\Ran_{x_0}}
$$
is a factorization $\on{Hecke}_{\Gcheck}^{\on{spec}}$-module space.

\sssec{}
Again, $\IndCoh^*$ is not defined on $\ncheck/\Bcheck_{\Ran_{x_0}} \underset{(\on{LS}_{\Gcheck}^{\on{mer}})_{\Ran_{x_0}}}{\times} \ncheck/\Bcheck_{\Ran_{x_0}}$, so we must define what we will mean by the notation 
$$
\IndCoh^*(\ncheck/\Bcheck_{\Ran_{x_0}} \underset{(\on{LS}_{\Gcheck}^{\on{mer}})_{\Ran_{x_0}}}{\times} \ncheck/\Bcheck_{\Ran_{x_0}}).
$$
However, we do not know of a parallel definition for the affine case to taking $L_{\nabla}^+G$-bi-coinvariance as in the spherical case. We will thus defer the definition of $\on{Aff}_{\Gcheck}^{\on{spec}}$ to Subsection \ref{Conservativity}.
%Fortunately, in \cite{GLC2} Appendix E.5, it is shown that the definition of $\on{Sph}_{\Gcheck}^{\on{spec}}$ given above coincides with applying a renormalization procedure to $\QCoh_{\on{co}}(\on{Hecke}_{\Gcheck}^{\on{spec}})$. We will perform the analogous renormalization procedure to 
%$$
%\QCoh_{\on{co}}(\ncheck/\Bcheck_{\Ran_{x_0}} \underset{(\on{LS}_{\Gcheck}^{\on{mer}})_{\Ran_{x_0}}}{\times} \ncheck/\Bcheck_{\Ran_{x_0}}).
%$$

%\sssec{}
%By definition, we take 
%$$
%\IndCoh^*(\ncheck/\Bcheck_{\Ran_{x_0}} \underset{(\on{LS}_{\Gcheck}^{\on{mer}})_{\Ran_{x_0}}}{\times} \ncheck/\Bcheck_{\Ran_{x_0}})
%$$
%to mean the sheaf of categories on $\Ran_{x_0}$ such that for a test affine scheme $S \to \Ran_{x_0}$,
%$$
%\IndCoh^*(\ncheck/\Bcheck_{\Ran_{x_0}} \underset{(\on{LS}_{\Gcheck}^{\on{mer}})_{\Ran_{x_0}}}{\times} \ncheck/\Bcheck_{\Ran_{x_0}})_S
%$$
%is the ind-completion of the subcategory of 
%$$
%\QCoh_{\on{co}}(\ncheck/\Bcheck_{\Ran_{x_0}} \underset{(\on{LS}_{\Gcheck}^{\on{mer}})_{\Ran_{x_0}}}{\times} \ncheck/\Bcheck_{\Ran_{x_0}})_S^{>-\infty}
%$$
%generated under finite colimits by the essential image of
%$$
%i_{*} : \QCoh(\ncheck/\Bcheck_{\Ran_{x_0}})_S \to \QCoh_{\on{co}}(\ncheck/\Bcheck_{\Ran_{x_0}} \underset{(\on{LS}_{\Gcheck}^{\on{mer}})_{\Ran_{x_0}}}{\times} \ncheck/\Bcheck_{\Ran_{x_0}})_S.
%$$

\begin{defn}
    The pair $(\on{Sph}_{\Gcheck}^{\on{spec}}, \on{Aff}_{\Gcheck}^{\on{spec}}) \in \FactModCat$ consists of the factorization category $\on{Sph}_{\Gcheck}^{\on{spec}} \in \FactAlgCat$ and the factorization $\on{Sph}_{\Gcheck}^{\on{spec}}$-module category
    $$
    \on{Aff}_{\Gcheck}^{\on{spec}} = \IndCoh^*(\ncheck/\Bcheck_{\Ran_{x_0}} \underset{(\on{LS}_{\Gcheck}^{\on{mer}})_{\Ran_{x_0}}}{\times} \ncheck/\Bcheck_{\Ran_{x_0}}).
    $$
    It has the structure of monoidal pair and action on $(\on{Rep}(\Gcheck), \QCoh(\ncheck/\Bcheck))$ defined analogously to the spherical situation.
\end{defn}

\subsection{Main Characters of IFLE} \label{Main Characters of IFLE}

\sssec{}
We begin this subsection by recalling the definition of the Kazhdan-Lusztig category as a factorization category from Section 2.1 of \cite{GLC2}.

\sssec{}
The factorization category of Kac-Moody modules, $\gmod_{\kappa}$, has a (strong) action of the factorization group $LG$ at level $\kappa$. For the construction of this factorization category and action, see Appendix B.14 of \cite{GLC2}. Thus, it makes sense to define

\begin{defn}
    The factorization category $\on{KL}(G)_{\kappa} \in \FactAlgCat$ is
    $$
        \gmod_{\kappa}^{L^+G}.
    $$
\end{defn}

\sssec{}
We can restrict the sheaf of categories $\gmod_{\kappa}$ along $\Ran_{x_0} \to \Ran$ to obtain $\gmod_{\kappa, \Ran_{x_0}}$. This is a factorization $\gmod_{\kappa}$-module category for tautological reasons. Furthermore, also for tautological reasons, there is an action of $LG_{\Ran_{x_0}}$ on $\gmod_{\kappa, \Ran_{x_0}}$. Thus we can take invariance for the factorization $L^+G$-module group $I_{\Ran_{x_0}}$ and obtain a factorization $\on{KL}(G)_{\kappa}$-module category that we will denote $\on{IKL}(G)_{\kappa}$.

\begin{defn}
    The pair $(\on{KL}(G)_{\kappa}, \on{IKL}(G)_{\kappa}) \in \FactModCat$ consists of the factorization category $\on{KL}(G)_{\kappa}$ and the factorization $\on{KL}(G)_{\kappa}$-module category
    $$
    \on{IKL}(G)_{\kappa} = \gmod_{\kappa, \Ran_{x_0}}^{I_{\Ran_{x_0}}}.
    $$
\end{defn}

\sssec{}
In the rest of the subsection and the paper, we will only be interested in the case $\kappa = \kappa_{\on{crit}}$.

\sssec{}
If $\mathcal{C}$ is a sheaf of categories on $\Ran$ with an action of $LG$ at level $\kappa_{\on{crit}}$, there is a monoidal action
$$
\on{Sph}_G \otimes \mathcal{C}^{L^+G} \longrightarrow \mathcal{C}^{L^+G}.
$$
Furthermore, if $\mathcal{D}$ is a sheaf of categories on $\Ran_{x_0}$ with an action of $LG$ at level $\kappa_{\on{crit}}$, there is a monoidal action
$$
\on{Aff}_G \otimes \mathcal{C}^{I} \longrightarrow \mathcal{C}^{I}.
$$
This allows us to consider the pair $(\on{KL}(G)_{\on{crit}}, \on{IKL}(G)_{\on{crit}})$ as having a monoidal action of the monoidal pair $(\on{Sph}_G, \on{Aff}_G)$.

\sssec{}
Now we will move on to the spectral counterparts. There is an affine $D$-scheme of $\Gcheck$-opers $\on{Op}_{\Gcheck}$ whose definition is provided in Section 3.1.3 of \cite{GLC2}. Then the horizontal arc space and horizontal loop space constructions allow us to define the factorization space of regular opers
$$
\on{Op}_{\Gcheck}^{\on{reg}} = L_{\nabla}^+\on{Op}_{\Gcheck}
$$
and the factorization space of meromorphic opers
$$
\on{Op}_{\Gcheck}^{\on{mer}} = L_{\nabla}\on{Op}_{\Gcheck}.
$$
There is a diagram
$$
\begin{tikzcd}
	{\on{Op}_{\Gcheck}^{\on{reg}}} && {\on{Op}_{\Gcheck}^{\on{mer}}} \\
	\\
	{\on{LS}_{\Gcheck}^{\on{reg}}} && {\on{LS}_{\Gcheck}^{\on{mer}}}
	\arrow[from=1-1, to=1-3]
	\arrow[from=1-1, to=3-1]
	\arrow[from=1-3, to=3-3]
	\arrow[from=3-1, to=3-3]
\end{tikzcd}
$$
which allows us to form the factorization space of monodromy-free opers
$$
\on{Op}_{\Gcheck}^{\on{mon-free}} = \on{Op}_{\Gcheck}^{\on{mer}} \underset{\on{LS}_{\Gcheck}^{\on{mer}}}{\times} \on{LS}_{\Gcheck}^{\on{reg}}.
$$

\begin{defn}
    The factorization category of Ind-coherent sheaves on monodromy-free opers is 
    $$
    \IndCoh^*(\on{Op}_{\Gcheck}^{\on{mon-free}}) \in \FactAlgCat.
    $$
    It has an action of $\on{Sph}_{\Gcheck}^{\on{spec}} \in \on{Alg}(\FactAlgCat)$ by convolution (the factorizable construction of this action is given in Appendix E.8 of \cite{GLC2}).
\end{defn}

\begin{rem}
For a more thorough discussion of monodromy-free opers, see Section 3 of \cite{GLC2}.
\end{rem}

\sssec{}
We can also take the pullback of $\on{Op}_{\Gcheck}^{\on{mer}}$ along $\Ran_{x_0} \to \Ran$ to obtain the factorization $\on{Op}_{\Gcheck}^{\on{mer}}$-module space
$$
\on{Op}_{\Gcheck, \on{Ran}_{x_0}}^{\on{mer}}.
$$
Then
$$
\on{Op}_{\Gcheck, \on{Ran}_{x_0}}^{\on{mer}} \underset{\on{LS}_{\Gcheck, \on{Ran}_{x_0}}^{\on{mer}}}{\times} \ncheck/\Bcheck_{\Ran_{x_0}}
$$
has the natural structure of factorization $\on{Op}_{\Gcheck}^{\on{mon-free}}$-module space.

This endows $\on{IndCoh}^*(\on{Op}_{\Gcheck, \on{Ran}_{x_0}}^{\on{mer}} \underset{\on{LS}_{\Gcheck, \on{Ran}_{x_0}}^{\on{mer}}}{\times} \ncheck/\Bcheck_{\Ran_{x_0}})$ with the structure of \emph{weak} factorization $\IndCoh^*(\on{Op}_{\Gcheck}^{\on{mon-free}})$-module category. However, since 
$$
\ncheck/\Bcheck_{\Ran_{x_0}} \longrightarrow \on{LS}_{\Gcheck, \on{Ran}_{x_0}}^{\on{mer}}
$$
is ind-proper, we have that
$$
\on{Op}_{\Gcheck, \on{Ran}_{x_0}}^{\on{mer}} \underset{\on{LS}_{\Gcheck, \on{Ran}_{x_0}}^{\on{mer}}}{\times} \ncheck/\Bcheck_{\Ran_{x_0}}
$$
is an ind-placid ind-scheme relative to $\Ran_{x_0}$, so the weak factorization module structure is in fact a factorization module structure (see Appendix B.13 of \cite{GLC2} for details).

\begin{defn}
    The pair $(\IndCoh^*(\on{Op}_{\Gcheck}^{\on{mon-free}}), \IndCoh^*(\on{Op}_{\Gcheck}^{\on{mer}} \underset{\on{LS}_{\Gcheck}^{\on{mer}}}{\times} \ncheck/\Bcheck)) \in \FactModCat$ consists of the factorization category $\IndCoh^*(\on{Op}_{\Gcheck}^{\on{mon-free}})$ and the factorization $\IndCoh^*(\on{Op}_{\Gcheck}^{\on{mon-free}})$-module category
    $$
    \IndCoh^*(\on{Op}_{\Gcheck}^{\on{mer}} \underset{\on{LS}_{\Gcheck}^{\on{mer}}}{\times} \ncheck/\Bcheck) = \IndCoh^*(\on{Op}_{\Gcheck, \Ran_{x_0}}^{\on{mer}} \underset{\on{LS}_{\Gcheck, \Ran_{x_0}}^{\on{mer}}}{\times} \ncheck/\Bcheck_{\Ran_{x_0}}).
    $$
    This pair has an action of the monoidal pair $(\on{Sph}_{\Gcheck}^{\on{spec}}, \on{Aff}_{\Gcheck}^{\on{spec}})$.
\end{defn}

\subsection{Approximation} \label{Approximation}
In this subsection, we will state and prove an approximation result, which will be used in later sections. The result is an extension of Corollary 4.5.7 of \cite{GLC2} from factorization categories to pairs (of a factorization category $\mathcal{C}$ and a factorization $\mathcal{C}$-module category $\mathcal{M}$ (at $x_0$)).

\sssec{}
We will begin by restating the setting of Section 4.5 in loc. cit.

Let $\mathcal{Y}$ be an affine $D$-scheme over $X$, and let $\mathcal{T}^+$ denote the factorization space $L_{\nabla}^+\mathcal{Y}$. If $\mathcal{T}$ is a factorization space that is an ind-placid ind-scheme, and
$$
\iota : \mathcal{T}^+ \to \mathcal{T}
$$
is a map of factorization spaces that extends to a unital-in-correspondences structure relative to $\mathcal{T}^+$ (as defined in Appendix C.10.6 of \cite{GLC2}), then
$$
\iota_* \mathcal{O}_{\mathcal{T}^+}
$$
is the factorization unit of the factorization category
$$
\QCoh_{\on{co}}(\mathcal{T}).
$$

\sssec{}
Let $\mathcal{Y}_0$ be a $D$-prestack such that $\mathcal{T}_0^+ = L_{\nabla}^+ \mathcal{Y}_0$ has an affine diagonal. Furthermore, let 
$$
f : \mathcal{Y} \to \mathcal{Y}_0
$$
be a map over $X_{dR}$ such that
$$
L^+f : \mathcal{T}^+ = L_{\nabla}^+ \mathcal{Y} \to L_{\nabla}^+ \mathcal{Y}_0 = \mathcal{T}_0^+
$$
extends to a map
$$
Lf : \mathcal{T} \to \mathcal{T}_0^+.
$$

\sssec{}
There is a lax unital functor of factorization categories
$$
\QCoh_{\on{co}}(\mathcal{T}) \overset{(Lf)_*}{\to} \QCoh_{\on{co}}(\mathcal{T}_0^+) \to \QCoh(\mathcal{T}_0^+)
$$
that sends the factorization unit $\iota_* \mathcal{O}_{\mathcal{T}^+}$ to the factorization
$$
(L^+f)_*\mathcal{O}_{\mathcal{T}^+}
$$
in $\QCoh(\mathcal{T}_0^+)$. Thus, the above functor factors through the unital factorization functor
$$
(Lf)_*^{\on{enh}} : \QCoh_{\on{co}}(\mathcal{T}) \to (L^+f)_*\mathcal{O}_{\mathcal{T}^+}\on{-FactMod} (\QCoh(\mathcal{T}_0^+)).
$$

Now we will recall the statement of Corollary 4.5.7 of loc. cit.

\begin{prop} \label{alg approximation} \cite{GLC2}
    Suppose that there is a map from an affine $D$-scheme
    $$
    \tilde{\mathcal{Y}}_0 \to \mathcal{Y}_0
    $$
    inducing an fpqc cover
    $$
    \tilde{\mathcal{T}}_0^+ = L_{\nabla}^+\tilde{\mathcal{Y}}_0 \to L_{\nabla}^+ \mathcal{Y}_0 = \mathcal{T}_0^+.
    $$
    Furthermore, suppose that there is an identification
    $$
    L_{\nabla}(\mathcal{Y} \underset{\mathcal{Y}_0}{\times} \tilde{\mathcal{Y}}_0) \underset{L_{\nabla}(\tilde{\mathcal{Y}}_0)}{\times} L_{\nabla}^+(\tilde{\mathcal{Y}}_0) \simeq \mathcal{T} \underset{\mathcal{T}_0^+}{\times} \tilde{\mathcal{T}}_0^+
    $$
    under which the base change of $\iota : \mathcal{T}^+ \to \mathcal{T}$ along the above fpqc cover,
    $$
    \mathcal{T}^+ \underset{\mathcal{T}_0^+}{\times} \tilde{\mathcal{T}}_0^+ \to \mathcal{T} \underset{\mathcal{T}_0^+}{\times} \tilde{\mathcal{T}}_0^+
    $$
    coincides with the map
    $$
    \mathcal{T}^+ \underset{\mathcal{T}_0^+}{\times} \tilde{\mathcal{T}}_0^+ = L_{\nabla}^+(\mathcal{Y} \underset{\mathcal{Y_0}}{\times} \tilde{\mathcal{Y}}_0) \to L_{\nabla}(\mathcal{Y} \underset{\mathcal{Y_0}}{\times} \tilde{\mathcal{Y}}_0) \underset{L_{\nabla}(\tilde{\mathcal{Y}}_0)}{\times} L_{\nabla}^+(\tilde{\mathcal{Y}}_0).
    $$
    Finally, suppose that the projection 
    $$
    \mathcal{Y} \underset{\mathcal{Y_0}}{\times} \tilde{\mathcal{Y}}_0 \to \tilde{\mathcal{Y}}_0
    $$
    is $D$-afp.
    Then the factorization functor
    $$
    (Lf)_*^{\on{enh}} : \QCoh_{\on{co}}(\mathcal{T}) \to (L^+f)_*\mathcal{O}_{\mathcal{T}^+}\on{-FactMod} (\QCoh(\mathcal{T}_0^+))
    $$
    induces an equivalence on the eventually coconnective subcategories
    $$
    \QCoh_{\on{co}}(\mathcal{T})^{>-\infty} \simeq [(L^+f)_*\mathcal{O}_{\mathcal{T}^+}\on{-FactMod} (\QCoh(\mathcal{T}_0^+))]^{>-\infty}.
    $$
\end{prop}

\sssec{}
We now detail our modification of the paradigm described above. Suppose the following strengthening of the assumptions of Corollary 4.5.7 in \cite{GLC2} holds:
$$
\mathcal{T} = L_{\nabla}\mathcal{Y} \underset{L_{\nabla}\mathcal{Y}_0}{\times} L_{\nabla}^+\mathcal{Y}_0.
$$

Let $\mathcal{S}_0$ be a factorization $\mathcal{T}_0$-module space equipped with a map
$$
\mathcal{S}_0 \to (L_{\nabla}\mathcal{Y}_0)_{\Ran_{x_0}}.
$$
Then
$$
(L_{\nabla}\mathcal{Y})_{\Ran_{x_0}} \underset{(L_{\nabla}\mathcal{Y}_0)_{\Ran_{x_0}}}{\times} \mathcal{S}_0
$$
is a factorization $\mathcal{T}$-module space, and we have a morphism of pairs
$$
(\QCoh_{\on{co}}(L_{\nabla}\mathcal{Y} \underset{L_{\nabla}\mathcal{Y}_0}{\times} L_{\nabla}^+\mathcal{Y}_0), \QCoh_{\on{co}}((L_{\nabla}\mathcal{Y})_{\Ran_{x_0}} \underset{(L_{\nabla}\mathcal{Y}_0)_{\Ran_{x_0}}}{\times} \mathcal{S}_0)) \to (\QCoh(L_{\nabla}^+\mathcal{Y}_0), \QCoh(\mathcal{S}_0))
$$
which enhances to
\begin{align}
&(\QCoh_{\on{co}}(L_{\nabla}\mathcal{Y} \underset{L_{\nabla}\mathcal{Y}_0}{\times} L_{\nabla}^+\mathcal{Y}_0), \QCoh_{\on{co}}((L_{\nabla}\mathcal{Y})_{\Ran_{x_0}} \underset{(L_{\nabla}\mathcal{Y}_0)_{\Ran_{x_0}}}{\times} \mathcal{S}_0)) \to \\ &(L^+f)_*\mathcal{O}_{\mathcal{T}^+}\on{-FactMod}(\QCoh(L_{\nabla}^+\mathcal{Y}_0), \QCoh(\mathcal{S}_0))
\end{align}

We claim the following
\begin{thm} \label{approximation}
    Suppose that $\mathcal{Y}_0 = \on{pt}/H \times X_{dR}$ for an affine algebraic group $H$. Then 
    \begin{align}
    &(\QCoh_{\on{co}}(L_{\nabla}\mathcal{Y} \underset{L_{\nabla}\mathcal{Y}_0}{\times} L_{\nabla}^+\mathcal{Y}_0), \QCoh_{\on{co}}((L_{\nabla}\mathcal{Y})_{\Ran_{x_0}} \underset{(L_{\nabla}\mathcal{Y}_0)_{\Ran_{x_0}}}{\times} \mathcal{S}_0)) \to \\ &(L^+f)_*\mathcal{O}_{\mathcal{T}^+}\on{-FactMod}(\QCoh(L_{\nabla}^+\mathcal{Y}_0), \QCoh(\mathcal{S}_0))
    \end{align}
    induces an equivalence on the eventually coconnective subcategories for each component.
\end{thm}
\begin{proof}
    In the case $H = \{e\}$, the claim is obvious. In general, we take $\tilde{\mathcal{Y}}_0 = \on{pt} \times X_{dR}$ and write
    $$
    \tilde{\mathcal{S}}_0 = (L_{\nabla}\tilde{\mathcal{Y}}_0)_{\Ran_{x_0}} \underset{(L_{\nabla}\mathcal{Y}_0)_{\Ran_{x_0}}}{\times} \mathcal{S}_0.
    $$
    Furthermore, denote by
    $$
    \tilde{\mathcal{S}}_0^{\bullet}
    $$
    the Cech nerve of $\tilde{\mathcal{S}}_0 \to \mathcal{S}_0$.
    In the notation of loc. cit. we have
    $$
    \QCoh_{\on{co}}((L_{\nabla}\mathcal{Y})_{\Ran_{x_0}} \underset{(L_{\nabla}\mathcal{Y}_0)_{\Ran_{x_0}}}{\times} \mathcal{S}_0) \simeq \on{Tot}(\QCoh_{\on{co}}((L_{\nabla}\tilde{\mathcal{Y}}^{\bullet})_{\Ran_{x_0}} \underset{(L_{\nabla}\tilde{\mathcal{Y}}_0^{\bullet})_{\Ran_{x_0}}}{\times} \tilde{\mathcal{S}}_0^{\bullet})
    $$
    and
    $$
    (L^+f)_*\mathcal{O}_{\mathcal{T}^+}\on{-FactMod}(\QCoh(\mathcal{S}_0)) \simeq \on{Tot}((L^+\tilde{f}^{\bullet})_*\mathcal{O}_{\tilde{\mathcal{T}}^{+, \bullet}}\on{-FactMod}(\QCoh(\tilde{\mathcal{S}}_0^{\bullet}))).
    $$
    Therefore, the claim follows from the $H = \{e\}$ case (and pullbacks thereof).
    
\end{proof}

\section{Arkhipov-Bezrukavnikov Equivalence} \label{Section AB}

\subsection{Fusability} \label{Fusability}
In this subsection, we will discuss the notion of a \emph{fusable} factorization $\on{Rep}(\Gcheck)$-module category. This notion will be defined carefully in the upcoming paper of Bogdanova, Chen, and Gaitsgory \cite{BCG}.

\sssec{}
Consider
$$
\on{LS}_{\Gcheck, x_0}^{\on{mer}} = \on{LS}_{\Gcheck}(D_{x_0}^{\times}).
$$
Then we can form the ($2$-, by which we mean $(\infty, 2)$-)category of module categories over the symmetric monoidal category $\QCoh(\on{LS}_{\Gcheck}(D_{x_0}^{\times}))$ and denote it
$$
\QCoh(\on{LS}_{\Gcheck}(D_{x_0}^{\times}))\on{-ModCat}.
$$
In Section 3.1 of \cite{Katia}, Bogdanova constructs the ($2$-)functor
$$
\on{Fact} : \QCoh(\on{LS}_{\Gcheck}(D_{x_0}^{\times}))\on{-ModCat} \longrightarrow \on{Rep}(\Gcheck)\on{-FactModCat}_{x_0}
$$
given by
$$
\boldsymbol{M} \mapsto \QCoh(\on{LS}_{\Gcheck}(D_{\Ran_{x_0}}^{\times})) \underset{\QCoh(\on{LS}_{\Gcheck}(D_{x_0}^{\times}))}{\otimes} \boldsymbol{M}.
$$
This is conjectured to be fully faithful, and in loc. cit. the following progress is made towards this conjecture:

\sssec{}
In Section 3.3 of \cite{Katia}, the definition of the stack of restricted local systems $\on{LS}_{\Gcheck}^{\on{restr}}(D_{x_0}^{\times})$ is given, with a map
$$
\on{LS}_{\Gcheck}^{\on{restr}}(D_{x_0}^{\times}) \longrightarrow \on{LS}_{\Gcheck}(D_{x_0}^{\times})
$$
which defines a monoidal functor
$$
\QCoh(\on{LS}_{\Gcheck}(D_{x_0}^{\times})) \longrightarrow \QCoh(\on{LS}_{\Gcheck}^{\on{restr}}(D_{x_0}^{\times}))
$$
via pullback, which in turn gives rise to
$$
\QCoh(\on{LS}_{\Gcheck}^{\on{restr}}(D_{x_0}^{\times}))\on{-ModCat} \longrightarrow \QCoh(\on{LS}_{\Gcheck}(D_{x_0}^{\times}))\on{-ModCat}
$$
via restriction.

Denoting the composition
$$
\QCoh(\on{LS}_{\Gcheck}^{\on{restr}}(D_{x_0}^{\times}))\on{-ModCat} \longrightarrow \QCoh(\on{LS}_{\Gcheck}(D_{x_0}^{\times}))\on{-ModCat} \overset{\on{Fact}}{\longrightarrow} \on{Rep}(\Gcheck)\on{-FactModCat}_{x_0} 
$$
by $\on{Fact}^{\on{restr}}$, Bogdanova proves in Theorem 3.8 of loc. cit. the following
\begin{prop}
    $\on{Fact}^{\on{restr}}$ is fully faithful. Moreover, the diagram
    $$
    \begin{tikzcd}
        {\QCoh(\on{LS}_{\Gcheck}^{\on{restr}}(D_{x_0}^{\times}))\on{-ModCat}} && {\on{Rep}(\Gcheck)\on{-FactModCat}_{x_0}} \\
	\\
	& {\on{DGCat}}
	\arrow["{\on{Fact}^{\on{restr}}}", from=1-1, to=1-3]
	\arrow["{\on{oblv}}"', from=1-1, to=3-2]
	\arrow["{\on{cores}_{x_0}}", from=1-3, to=3-2]
    \end{tikzcd}
    $$
    commutes.
\end{prop}

\sssec{}
Here is a definition of fusability.

\begin{defn}
    A factorization $\on{Rep}(\Gcheck)$-module category $\mathcal{M}$ is fusable if it is in the essential image of the functor
    $$
    \on{Fact}^{\on{restr}} : \QCoh(\on{LS}_{\Gcheck}^{\on{restr}}(D_{x_0}^{\times}))\on{-ModCat} \longrightarrow \on{Rep}(\Gcheck)\on{-FactModCat}_{x_0}.
    $$
    We will denote the full subcategory of $\on{Rep}(\Gcheck)\on{-FactModCat}_{x_0}$ consisting of fusable objects by
    $$
    \on{Rep}(\Gcheck)\on{-FactModCat}_{x_0}^{\on{fus}} \subset \on{Rep}(\Gcheck)\on{-FactModCat}_{x_0}.
    $$
\end{defn}

\begin{rem}
    In \cite{BCG}, a combinatorial characterization of fusability will appear, which will make it possible to show in practice that particular factorization $\on{Rep}(\Gcheck)$-module categories are fusable. Let us give a sketch of this characterization now:

    Given a factorization $\on{Rep}(\Gcheck)$-module category $\mathcal{M}$, we can consider its restriction to one copy of the curve, $\mathcal{M}_X$. Denoting
    $$
    U \overset{j}{\hookrightarrow} X \overset{i}{\hookleftarrow} x_0,
    $$
    we note that the further restriction to $U$ is
    $$
    \mathcal{M}_U \simeq \on{Rep}(\Gcheck) \otimes \underline{\Dmod}(U) \otimes \mathcal{M}_{x_0}.
    $$

    There is a functor
    $$
    j_{\bullet} : \on{Rep}(\Gcheck) \otimes \Dmod(U) \otimes \mathcal{M}_{x_0} \simeq \Gamma(U, \mathcal{M}_U) \longrightarrow \Gamma(X, \mathcal{M}_X)
    $$
    and a functor
    $$
    i^{\bullet} : \Gamma(X, \mathcal{M}_X) \longrightarrow \on{Pro}(\mathcal{M}_{x_0}).
    $$
    We can restrict the composition
    $$
    i^{\bullet} j_{\bullet} : \on{Rep}(\Gcheck) \otimes \Dmod(U) \otimes \mathcal{M}_{x_0} \to \on{Pro}(\mathcal{M}_{x_0})
    $$
    to
    $$
    \on{Rep}(\Gcheck) \otimes \on{QLisse}(U) \otimes \mathcal{M}_{x_0} \subset \on{Rep}(\Gcheck) \otimes \Dmod(U) \otimes \mathcal{M}_{x_0}.
    $$

    Then the condition of $\mathcal{M}$ being \emph{fusable} is equivalent to
    $$
    \on{Rep}(\Gcheck) \otimes \on{QLisse}(U) \otimes \mathcal{M}_{x_0} \longrightarrow \on{Pro}(\mathcal{M}_{x_0})
    $$
    landing in $\mathcal{M}_{x_0}$ instead of its pro-category (plus some conditions having to do with higher powers of the curve, which will appear in \cite{BCG}).
\end{rem}

\sssec{} \label{useful fusable}
The usefulness of this notion is twofold. First, because
$$
\on{oblv} : \QCoh(\on{LS}_{\Gcheck}^{\on{restr}}(D_{x_0}^{\times}))\on{-ModCat} \longrightarrow \on{DGCat}
$$
is conservative, if $F : \mathcal{M}_1 \longrightarrow \mathcal{M}_2$ is a factorization functor between fusable factorization $\on{Rep}(\Gcheck)$-module categories, $F$ is an equivalence if $F_{x_0}$ is an equivalence.

Second, if $\mathcal{M}$ is factorization $\on{Rep}(\Gcheck)$-module category, there is a factorization $\on{Rep}(\Gcheck)$-submodule category
$$
\mathcal{M}_{\on{QLisse}} \subset \mathcal{M}
$$
characterized as follows.

Let $U = X \setminus \{x_0\}$ and for each surjective map $n \to m$, let $\Delta_{n \to m} : X^m \longrightarrow X^n$ be the diagonal. $\Delta_n$ will denote the union of all (non-identity) diagonals in $X^n$. We think of $X^n$ (and its subsets) as mapping to $\Ran_{x_0}$ via

$$
(x_1, ..., x_n) \mapsto \{x_0, x_1, ..., x_n\}.
$$

We have by factorization
$$
\mathcal{M}|_{U^n \setminus \Delta_n} \simeq (\on{Rep}(\Gcheck)^{\otimes n} \otimes \Dmod(U)^{\otimes n} \otimes \mathcal{M}_{x_0})|_{U^n \setminus \Delta_n}
$$
and we define $\mathcal{M}_{\on{QLisse}}$ by
$$
\mathcal{M}_{\on{QLisse}}|_{U^n \setminus \Delta_n} \simeq (\on{Rep}(\Gcheck)^{\otimes n} \otimes \on{QLisse}(U)^{\otimes n} \otimes \mathcal{M}_{x_0})|_{U^n \setminus \Delta_n}.
$$

For $\mathcal{M}$ fusable, it is the case that constructing a factorization functor between factorization $\on{Rep}(\Gcheck)$-module categories with source $\mathcal{M}_{\on{QLisse}}$ is sufficient to construct one with source $\mathcal{M}$.

\sssec{}
In the next subsection, we will use this latter fact to construct the morphism of pairs $\on{AB}$, and in the following subsection, we will use the former observation to argue that it is an equivalence of pairs. To do so, we will need the following

\begin{prop}
    $\on{Whit}^!(\on{Fl}_G)$ is fusable when thought of as a factorization $\on{Rep}(\Gcheck)$-module category via the equivalence of factorization categories $\on{CS}_G$.
\end{prop}
\begin{proof}
    In \cite{Braiding}, the family $\on{Fl}' \to X$ is studied, which is the factorization module space $\on{Fl}_{G, \Ran_{x_0}}$ restricted to one copy of the curve. Using the convolution product
    $$
    \Dmod(\on{Gr}_G)^{L^+G} \otimes \Dmod(U) \otimes \Dmod(\on{Fl}_G) \longrightarrow \Dmod(\on{Fl}_G)
    $$
    we obtain our desired functor $i^{\bullet}j_{\bullet}$ via the (pointwise) naive Satake functor $\on{Rep}(\Gcheck) \to \Dmod(\on{Gr}_G)^{L^+G}$. Thus, $\Dmod(\on{Fl}_G)$ is fusable, and after taking $\on{Whit}^!$ of both sides, so is $\on{Whit}^!(\on{Fl}_G)$.
\end{proof}

\sssec{}
Furthermore, by construction,
\begin{prop}
    $\on{QCoh}(\ncheck/\Bcheck)$ is fusable.
\end{prop}

\begin{proof}
    Unwinding our construction of $\on{QCoh}(\ncheck/\Bcheck)$ and comparing it with the description of $\on{Fact}^{\on{restr}}$ in \cite{Katia} makes it clear that $\on{QCoh}(\ncheck/\Bcheck) \simeq \on{Fact}^{\on{restr}}(\on{QCoh}(\ncheck/\Bcheck)_{x_0})$.
\end{proof}

\subsection{Construction of AB} \label{Construction of AB}
In this subsection, we will explain how to construct the morphism of pairs
$$
\on{AB} : ( \on{Whit}^!(G), \on{Whit}^!(\on{Fl}_G)) \longrightarrow (\on{Rep}(\Gcheck), \on{QCoh}(\ncheck/\Bcheck))
$$
in $\FactModCat$.

\sssec{}
We will use the pair
$$
(\on{Whit}^!(G), \on{Whit}^!(LG)^{\frac{\infty}{2}})
$$
as an intermediary. Here
$$
\on{Whit}^!(LG)^{\frac{\infty}{2}} = \on{Whit}^!(LG)^{LNL^+T}
$$
where by $LNL^+T$ we mean the factorization $L^+G$-module (ind-)group at $x_0$.

\sssec{}
There is a functor
$$
(\on{Whit}^!(G), \on{Whit}^!(\on{Fl}_G)) \longrightarrow (\on{Whit}^!(G), \on{Whit}^!(LG)^{\frac{\infty}{2}})
$$
that is a straightforward upgrade of the pointwise version (which is an equivalence)
$$
\on{Whit}^!(\on{Fl}_G)_{x_0} \overset{\simeq}{\longrightarrow} \on{Whit}^!(LG)_{x_0}^{\frac{\infty}{2}}
$$
that appears (more precisely, as the inverse, thus continuous left adjoint, of the functor) in Corollary 6.2.3 of \cite{CPSC2}.

\sssec{}
We will adapt the construction from Appendix B of \cite{GLC3} (where the \emph{factorization categories} $\on{Whit}^!(LG)^{\frac{\infty}{2}}$ and $\IndCoh(\on{LS}_{\Bcheck}^{\on{mer}} \underset{\on{LS}_{\Tcheck}^{\on{mer}}}{\times} \on{LS}_{\Bcheck}^{\on{reg}})$ are related) to the setting of factorization module categories in order to construct a map
$$
(\on{Whit}^!(G), \on{Whit}^!(LG)^{\frac{\infty}{2}}) \longrightarrow (\on{Rep}(\Gcheck), \on{QCoh}(\ncheck/\Bcheck)).
$$

\sssec{}
On the geometric side, consider the functor
$$
\on{Whit}^!(LG)^{\frac{\infty}{2}} \longrightarrow \on{Whit}^!(T) \simeq \Dmodh(\on{Gr}_T)
$$
given by
\begin{align}
    \on{Whit}^!(LG)^{\frac{\infty}{2}} &\subset \Dmodh(LG)^{LNL^+T} \\
    &\simeq (\Dmodh(LG) \otimes \Dmodh(\on{Gr}_T))^{LB} \\
    &\to \Dmodh(LG) \otimes \Dmodh(\on{Gr}_T) \\
    &\overset{1_{LG}^!}{\to} \Dmodh(\on{Gr}_T) \overset{\on{shift}}{\to} \Dmodh(\on{Gr}_T).
\end{align}
This has the structure of being a morphism of pairs
$$
(\on{Whit}^!(G), \on{Whit}^!(LG)^{\frac{\infty}{2}}) \longrightarrow (\on{Whit}^!(T), \on{Whit}^!(T))
$$
where the first component $\on{Whit}^!(G) \to \on{Whit}^!(T)$ is the Jacquet functor $J_{\on{Whit}}^!$ defined in \cite{GLC3}.

This map factors through
$$
(\on{Whit}^!(G), \on{Whit}^!(LG)^{\frac{\infty}{2}}) \longrightarrow \Omega \on{-FactModCat}(\on{Whit}^!(T), \on{Whit}^!(T)).
$$
Abusively denoting the factorization restriction of $\Omega \on{-FactModCat}(\on{Whit}^!(T))$ along $\on{Whit}^!(G) \to \Omega \on{-FactModCat}(\on{Whit}^!(T))$ by
$$
(\on{Whit}^!(G), \Omega \on{-FactMod}(\on{Whit}^!(T))),
$$
we obtain a morphism
$$
(\on{Whit}^!(G), \on{Whit}^!(LG)^{\frac{\infty}{2}}) \longrightarrow (\on{Whit}^!(G), \Omega \on{-FactMod}(\on{Whit}^!(T)))
$$
which is an equivalence when restricted to $x_0 \in \Ran_{x_0}$.

\sssec{}
On the spectral side, we can consider the functor
$$
\on{QCoh}(\ncheck/\Bcheck) \longrightarrow \on{Rep}(\Tcheck)
$$
given by pushforward along
$$
\on{LS}_{\Bcheck}(D^{\times}) \underset{\on{LS}_{\Tcheck}(D^{\times})}{\times} \on{LS}_{\Tcheck}(D) \longrightarrow \on{LS}_{\Tcheck}(D).
$$
This has the structure of being a morphism of pairs
$$
(\on{Rep}(\Gcheck), \on{QCoh}(\ncheck/\Bcheck)) \longrightarrow (\on{Rep}(\Tcheck), \on{Rep}(\Tcheck))
$$
where the first component $\on{Rep}(\Gcheck) \to \on{Rep}(\Tcheck)$ is given by restriction to $\on{Rep}(\Bcheck)$ and then taking invariants for $\check{N}$.

This map factors through
$$
(\on{Rep}(\Gcheck), \on{QCoh}(\ncheck/\Bcheck)) \longrightarrow \Omega^{\on{spec}}\on{-FactMod}(\on{Rep}(\Tcheck), \on{Rep}(\Tcheck))
$$
which factors through
$$
(\on{Rep}(\Gcheck), \on{QCoh}(\ncheck/\Bcheck)) \longrightarrow (\on{Rep}(\Gcheck), \Omega^{\on{spec}}\on{-FactMod}(\on{Rep}(\Tcheck))).
$$
This morphism is an equivalence when restricted to $x_0 \in \Ran_{x_0}$. Furthermore, it is also an equivalence on the eventually coconnective part of the second component, because so is the factorization functor (in $\FactAlgCat$)
$$
\on{IndCoh}(\on{LS}_{\Bcheck}^{\on{mer}} \underset{\on{LS}_{\Tcheck}^{\on{mer}}}{\times} \on{LS}_{\Bcheck}^{\on{reg}}) \longrightarrow \Omega^{\on{spec}}\on{-FactMod}(\on{Rep}(\Tcheck))
$$
by Proposition \ref{alg approximation}.

\sssec{}
We may identify
$$
(\on{Whit}^!(G), \Omega \on{-FactMod}(\on{Whit}^!(T))) \simeq (\on{Rep}(\Gcheck), \Omega^{\on{spec}}\on{-FactMod}(\on{Rep}(\Tcheck)))
$$
since the matching of $\Omega$ with $\Omega^{\on{spec}}$ follows from the diagram
$$
\begin{tikzcd}
    {\on{Whit}^!(G)} && {\on{Rep}(\Gcheck)} \\
	\\
	{\on{Whit}^!(T)} && {\on{Rep}(\Tcheck).}
	\arrow["{\on{CS}_G}", from=1-1, to=1-3]
	\arrow[from=1-1, to=3-1]
	\arrow[from=1-3, to=3-3]
	\arrow["{\on{CS}_T}", from=3-1, to=3-3]
\end{tikzcd}
$$

\sssec{}
We would like to finally construct the promised morphism
$$
\on{AB}: (\on{Whit}^!(G), \on{Whit}^!(\on{Fl}_G)) \longrightarrow (\on{Rep}(\Gcheck), \on{QCoh}(\ncheck/\Bcheck))
$$
by saying that the compact generators of $\on{Whit}^!(\on{Fl}_G)$ are sent to eventually coconnective objects in $\Omega^{\on{spec}}\on{-FactMod}(\on{Rep}(\Tcheck))$ (or, equivalently, in $\on{Rep}(\Tcheck)$). The trouble is that $\on{Whit}^!(\on{Fl}_G)$ is probably not compactly generated. However, we are rescued by the fusability of $\on{Whit}^!(\on{Fl}_G)$, since
$$
\on{Whit}^!(\on{Fl}_G)_{\on{QLisse}} \subset \on{Whit}^!(\on{Fl}_G)
$$
\emph{is} compactly generated, and it is straightforward to check that its compact generators are sent to eventually coconnectives. Thus, we obtain the desired morphism of pairs.

\subsection{AB is a Pointwise Equivalence} \label{AB is a Pointwise Equivalence}
\sssec{}
Observe that 
\begin{align}
\on{Whit}^!(\on{Fl}_G)_{x_0} \overset{\simeq}{\to} \on{Whit}^!(LG)_{x_0}^{\frac{\infty}{2}} &\overset{\simeq}{\to} \Omega\on{-FactMod}(\on{Whit}^!(T))_{x_0} \\ &\simeq \Omega^{\on{spec}}\on{-FactMod}(\on{Rep}(\Tcheck))_{x_0} \overset{\simeq}{\longleftarrow} \on{QCoh}(\ncheck/\Bcheck)_{x_0}
\end{align}
so 
$$
\on{AB}_{x_0} : \on{Whit}^!(\on{Fl}_G)_{x_0} \overset{\simeq}{\longrightarrow} \on{QCoh}(\ncheck/\Bcheck)_{x_0}
$$
is an equivalence of categories. Since $\on{Whit}^!(\on{Fl}_G)$ and $\on{QCoh}(\ncheck/\Bcheck)$ are fusable, we can conclude that $\on{AB}$ is an equivalence in $\FactModCat$, proving Theorem \ref{Theorem AB}.

\sssec{}
However, let us make another observation that is tautological but will be useful in the sequel. In \cite{AB}, the authors construct an equivalence of categories (up to renormalization)
$$
\on{AB}^{\on{classical}} : \on{Whit}^!(\on{Fl}_G)_{x_0} \simeq \on{QCoh}(\ncheck/\Bcheck)_{x_0}.
$$
$\on{AB}_{x_0}$ and $\on{AB}^{\on{classical}}$ may or may not be isomorphic, but since both functors are equivalences, we may write
$$
\alpha = \on{AB}_{x_0} \circ \on{AB}^{\on{classical}, -1}
$$
for the automorphism of $\on{QCoh}(\ncheck/\Bcheck)_{x_0}$ and
$$
\beta = \on{AB}^{\on{classical}, -1} \circ \on{AB}_{x_0}
$$
for the automorphsim of $\on{Whit}^!(\on{Fl}_G)_{x_0}$.

Then we have
$$
\on{AB}_{x_0} = \alpha \circ \on{AB}^{\on{classical}}
$$
and
$$
\on{AB}_{x_0} = \on{AB}^{\on{classical}} \circ \beta.
$$

\section{Bezrukavnikov's Equivalence} \label{Section B}
\subsection{Conservativity} \label{Conservativity}
In this subsection, we will describe our approach to working with factorization $\on{Sph}_{\Gcheck}^{\on{spec}}$-module categories.

\sssec{}
Let $\mathcal{C}_1 \to \mathcal{C}_2$ be a morphism in $\FactAlgCat$. Consider the (2-)functor of \emph{factorization restriction}
$$
\on{FactRes}_{\mathcal{C}_1 \to \mathcal{C}_2} : \mathcal{C}_2 \on{-FactModCat}_{x_0} \longrightarrow \mathcal{C}_1 \on{-FactModCat}_{x_0}
$$
defined in \cite{CF}. We will show that this functor is conservative for the factorization functor $\on{Rep}(\Gcheck) \to \on{Sph}_{\Gcheck}^{\on{spec}}$, but to do so, we will first need this

\begin{lem}
    Suppose that we have a space $\mathcal{X}$ such that $\mathcal{X}_{dR}$ is 1-affine. Furthermore, suppose 
    $$
    \mathcal{X} = \mathcal{U} \sqcup \mathcal{Z}
    $$
    where $\mathcal{U}$ is open and $\mathcal{Z}$ is closed, and assume that both $\mathcal{U}_{dR}$ and $\mathcal{Z}_{dR}$ are also 1-affine.
    Denote 
    $$
    \on{CrysCat}(\mathcal{X}) = \on{ShvCat}(\mathcal{X}_{dR}).
    $$
    Consider the assignment
    \begin{align}
    \on{CrysCat}(\mathcal{X}) &\longrightarrow (\on{CrysCat}(\mathcal{U}) \times \on{CrysCat}(\mathcal{Z})) \underset{\on{DGCat^{\times 2}}}{\times} \on{DGCat}^{\Delta^1} \\
    \mathcal{C} &\longmapsto (\mathcal{C}_U, \mathcal{C}_Z, i^*j_*).
    \end{align}
    This is conservative.
\end{lem}
\begin{proof}
    By recollement.
\end{proof}

\begin{prop}
    Let
    $$
    \on{nv} : \on{Rep}(\Gcheck) \longrightarrow \on{Sph}_{\Gcheck}^{\on{spec}}
    $$
    denote the left adjoint constructed in Appendix B.2.13 of \cite{GLC2}. 
    
    Factorization restriction along $\on{nv}$
    $$
    \on{FactRes}_{\on{Rep}(\Gcheck) \to \on{Sph}_{\Gcheck}^{\on{spec}}} : \on{Sph}_{\Gcheck}^{\on{spec}} \on{-FactModCat}_{x_0} \longrightarrow \on{Rep}(\Gcheck) \on{-FactModCat}_{x_0}
    $$
    is conservative.
\end{prop}
\begin{proof}
    Let 
    $$
    F : \mathcal{M} \longrightarrow \mathcal{N}
    $$
    be a morphism of factorization $\on{Sph}_{\Gcheck}^{\on{spec}}$-module categories such that 
    $$
    \on{FactRes}_{\on{Rep}(\Gcheck) \to \on{Sph}_{\Gcheck}^{\on{spec}}}(F) : \on{FactRes}_{\on{Rep}(\Gcheck) \to \on{Sph}_{\Gcheck}^{\on{spec}}}(\mathcal{M}) \longrightarrow \on{FactRes}_{\on{Rep}(\Gcheck) \to \on{Sph}_{\Gcheck}^{\on{spec}}} (\mathcal{N})
    $$
    is an equivalence of factorization $\on{Rep}(\Gcheck)$-module categories.
    We want to show that $F$ is an equivalence. It suffices to show that it is an equivalence as a morphism of sheaves of categories on $\Ran_{x_0}$, and to do so, it suffices to show that for each $n \in \mathbb{N}$,
    $$
    F_{X^n} : \mathcal{M}_{X^n} \longrightarrow \mathcal{N}_{X^n}
    $$
    is an equivalence.
    
    Let us proceed inductively: the $n = 0$ case is immediate since
    $$
    F_{x_0} : \mathcal{M}_{x_0} \longrightarrow \mathcal{N}_{x_0}
    $$
    just identifies with
    $$
    \on{FactRes}_{\on{Rep}(\Gcheck) \to \on{Sph}_{\Gcheck}^{\on{spec}}}(F)_{x_0} : \on{FactRes}_{\on{Rep}(\Gcheck) \to \on{Sph}_{\Gcheck}^{\on{spec}}}(\mathcal{M})_{x_0} \overset{\simeq}{\longrightarrow} \on{FactRes}_{\on{Rep}(\Gcheck) \to \on{Sph}_{\Gcheck}^{\on{spec}}} (\mathcal{N})_{x_0}.
    $$
    Next, we will assume the inductive hypothesis for all $m < n$ and prove it for $n$. Denoting $U = X \setminus x_0$, note that
    $$
    X^n = (U^n \setminus \Delta_n) \sqcup (\Delta_n \cup (X^n \setminus U^n)).
    $$
    By factorization,
    $$
    F_{U^n \setminus \Delta_n} : \mathcal{M}_{U^n \setminus \Delta_n} \longrightarrow \mathcal{N}_{U^n \setminus \Delta_n}
    $$
    identifies with the equivalence
    $$
    \on{id}_{\on{Sph}_{\Gcheck}^{\on{spec}}} \otimes F_{x_0} : \on{Sph}_{\Gcheck}^{\on{spec} \otimes n} \otimes \Dmod(U^n \setminus \Delta_n) \otimes \mathcal{M}_{x_0} \overset{\simeq}{\longrightarrow}  \on{Sph}_{\Gcheck}^{\on{spec} \otimes n} \otimes \Dmod(U^n \setminus \Delta_n) \otimes \mathcal{N}_{x_0}.
    $$
    On the other hand, $\Delta_n \cup (X^n \setminus U^n)$ is a union of hyperplanes $X^{n-1}$ (in a way where the maps to $\Ran_{x_0}$ are compatible), so by the inductive hypothesis,
    $$
    F_{\Delta_n \cup (X^n \setminus U^n)} : \mathcal{M}_{\Delta_n \cup (X^n \setminus U^n)} \longrightarrow \mathcal{N}_{\Delta_n \cup (X^n \setminus U^n)}
    $$
    is an equivalence.
    Invoking the above lemma, it suffices to show that the lax commutative diagram
    $$
    \begin{tikzcd}
        {\mathcal{N}_{U^n \setminus \Delta_n}} && {\mathcal{N}_{\Delta_n \cup (X^n \setminus U^n)}} \\
	\\
	{\mathcal{M}_{U^n \setminus \Delta_n}} && {\mathcal{M}_{\Delta_n \cup (X^n \setminus U^n)}}
	\arrow["{i^*j_*}", from=1-1, to=1-3]
	\arrow["\simeq", from=3-1, to=1-1]
	\arrow["{i^*j_*}"', from=3-1, to=3-3]
	\arrow[between={0.4}{0.6}, Rightarrow, from=3-3, to=1-1]
	\arrow["\simeq"', from=3-3, to=1-3]
    \end{tikzcd}
    $$
    is strictly commutative. However, this follows from the cube
    $$
    \begin{tikzcd}
        & {\on{FactRes}(\mathcal{N})_{U^n \setminus \Delta_n}} && {\on{FactRes}(\mathcal{N})_{\Delta_n \cup (X^n \setminus U^n)}} \\
	{\on{FactRes}(\mathcal{M})_{U^n \setminus \Delta_n}} && {\on{FactRes}(\mathcal{M})_{\Delta_n \cup (X^n \setminus U^n)}} \\
	& {\mathcal{N}_{U^n \setminus \Delta_n}} && {\mathcal{N}_{\Delta_n \cup (X^n \setminus U^n)}} \\
	{\mathcal{M}_{U^n \setminus \Delta_n}} && {\mathcal{M}_{\Delta_n \cup (X^n \setminus U^n)}}
	\arrow[from=1-2, to=1-4]
	\arrow[from=1-2, to=3-2]
	\arrow[from=1-4, to=3-4]
	\arrow["\simeq", from=2-1, to=1-2]
	\arrow[from=2-1, to=2-3]
	\arrow[from=2-1, to=4-1]
	\arrow["\simeq"', between={0.4}{0.6}, Rightarrow, from=2-3, to=1-2]
	\arrow["\simeq", from=2-3, to=1-4]
	\arrow[from=2-3, to=4-3]
	\arrow[from=3-2, to=3-4]
	\arrow["\simeq", from=4-1, to=3-2]
	\arrow[from=4-1, to=4-3]
	\arrow[between={0.4}{0.6}, Rightarrow, from=4-3, to=3-2]
	\arrow["\simeq", from=4-3, to=3-4]
    \end{tikzcd}
    $$
    because the top square commutes strictly and the vertical arrow
    $$
    \on{FactRes}_{\on{Rep}(\Gcheck) \to \on{Sph}_{\Gcheck}^{\on{spec}}}(\mathcal{M})_{U^n \setminus \Delta_n} \longrightarrow \mathcal{M}_{U^n \setminus \Delta_n}
    $$
    identifies with
    $$
    \on{Rep}(\Gcheck)^{\otimes n} \otimes \Dmod(U^n \setminus \Delta_n) \otimes \mathcal{M}_{x_0} \longrightarrow \on{Sph}_{\Gcheck}^{\on{spec} \otimes n} \otimes \Dmod(U^n \setminus \Delta_n) \otimes \mathcal{M}_{x_0}
    $$
    which generates the target.
\end{proof}

\begin{rem}
    Although it is not recorded anywhere in the literature, $\on{FactRes}$ is conservative more generally along any factorization functor $\mathcal{C} \longrightarrow \mathcal{D}$ with a continuous right adjoint. We will not prove this fact since it is not necessary for our argument.
\end{rem}

\sssec{}
Suppose 
$$
F : \mathcal{M}_1 \longrightarrow \mathcal{M}_2
$$
is a morphism in $\on{Sph}_{\Gcheck}^{\on{spec}}\on{-FactModCat}_{x_0}$.

If
$$
\on{FactRes}_{\on{Rep}(\Gcheck) \to \on{Sph}_{\Gcheck}^{\on{spec}}}(\mathcal{M}_1)
$$
and
$$
\on{FactRes}_{\on{Rep}(\Gcheck) \to \on{Sph}_{\Gcheck}^{\on{spec}}}(\mathcal{M}_2)
$$
are fusable as factorization $\on{Rep}(\Gcheck)$-module categories, then to show that $F$ is an equivalence, it suffices to show that
$$
\on{FactRes}_{\on{Rep}(\Gcheck) \to \on{Sph}_{\Gcheck}^{\on{spec}}}(F)_{x_0} = F_{x_0}
$$
is an equivalence.

\sssec{}
Furthermore, if $\mathcal{M}$ is a factorization $\on{Sph}_{\Gcheck}^{\on{spec}}$-module category, there is a factorization $\on{Sph}_{\Gcheck}^{\on{spec}}$-submodule category
$$
\mathcal{M}_{\on{Sph}_{\Gcheck}^{\on{spec}}\on{-QLisse}} \subset \mathcal{M}
$$
defined parallel to the $\on{Rep}(\Gcheck)$ case, i.e.
\begin{align}
\mathcal{M}_{\on{Sph}_{\Gcheck}^{\on{spec}}\on{-QLisse}}|_{U^n \setminus \Delta_n} &\simeq (\on{Sph}_{\Gcheck}^{\on{spec} \otimes n} \otimes \on{QLisse}(U)^{\otimes n} \otimes \mathcal{M}_{x_0})|_{U^n \setminus \Delta_n} \\
&\subset (\on{Sph}_{\Gcheck}^{\on{spec} \otimes n} \otimes \Dmod(U)^{\otimes n} \otimes \mathcal{M}_{x_0})|_{U^n \setminus \Delta_n} \simeq \mathcal{M}|_{U^n \setminus \Delta_n}.
\end{align}

If $\on{FactRes}_{\on{Rep}(\Gcheck) \to \on{Sph}_{\Gcheck}^{\on{spec}}}(\mathcal{M})$ is fusable, then constructing a factorization functor between factorization $\on{Sph}_{\Gcheck}^{\on{spec}}$-module categories with source $\mathcal{M}_{\on{Sph}_{\Gcheck}^{\on{spec}}\on{-QLisse}}$ is sufficient to construct one with source $\mathcal{M}$.

\sssec{}
Here is a result that will appear in \cite{BCG} that recontextualizes our discussion so far:

\begin{prop}
    Let
    $$
    \on{Inertia} = \mathcal{Y} \underset{\mathcal{Y} \times \mathcal{Y}}{\times} \mathcal{Y}
    $$
    denote the inertia stack of $\mathcal{Y} = \on{LS}_{\Gcheck}^{\on{restr}}(D_{x_0}^{\times})$.
    
    There is a fully faithful (2-)functor
    $$
    \on{Fact}^{\on{inertia}} : \on{QCoh}(\on{Inertia})\on{-ModCat} \longrightarrow \on{Sph}_{\Gcheck}^{\on{spec}}\on{-FactModCat}_{x_0}
    $$
    whose essential image consists precisely of factorization $\on{Sph}_{\Gcheck}^{\on{spec}}$-module categories whose factorization restriction to $\on{Rep}(\Gcheck)$ are fusable.
\end{prop}

\begin{defn}
    Since there are (obviously compatible) maps
    $$
    (\ncheck/\Bcheck)_{x_0} \longrightarrow \on{LS}_{\Gcheck}^{\on{restr}}(D_{x_0}^{\times})
    $$
    and
    $$
    \gcheck/\Gcheck \longrightarrow \on{LS}_{\Gcheck}^{\on{restr}}(D_{x_0}^{\times}) \longrightarrow \on{LS}_{\Gcheck}^{\on{restr}}(D_{x_0}^{\times}) \times \on{LS}_{\Gcheck}^{\on{restr}}(D_{x_0}^{\times}),
    $$
    we can define the factorization $\on{Sph}_{\Gcheck}^{\on{spec}}$-module category $\on{Aff}_{\Gcheck}^{\on{spec}}$ to be
    $$
    \on{Fact}^{\on{inertia}}(\on{IndCoh}((\ncheck/\Bcheck)_{x_0} \underset{\gcheck/\Gcheck}{\times} (\ncheck/\Bcheck)_{x_0})).
    $$
\end{defn}

%\sssec{}
%In the next subsection, we will use this latter fact to construct the morphism of pairs $\on{B}$, and in the following subsection, we will use the former observation to argue that it is an equivalence of pairs. To do so, we will need the following

\begin{prop}
    We may think of $\on{Aff}_{G}$ as a factorization $\on{Sph}_{\Gcheck}^{\on{spec}}$-module category via the equivalence of factorization categories $\on{Sat}_G$. Then $\on{FactRes}_{\on{Rep}(\Gcheck) \to \on{Sph}_{\Gcheck}^{\on{spec}}}(\on{Aff}_{G})$ is fusable.
\end{prop}
\begin{proof}
    As with $\on{Whit}^!(\on{Fl}_G)$, we agian use the fusability of $\Dmod(\on{Fl}_G)$, but this time we take invariants with respect to the Iwahori.
\end{proof}

\sssec{}
Furthermore, by construction,

\begin{prop}
    $\on{FactRes}_{\on{Rep}(\Gcheck) \to \on{Sph}_{\Gcheck}^{\on{spec}}}(\on{Aff}_{\Gcheck}^{\on{spec}})$ is fusable.
\end{prop}

\subsection{Construction of B} \label{Construction of B}
In this subsection, we will use the equivalence of pairs $\on{AB}$ to construct the morphism of pairs
$$
    \on{B} : (\on{Sph}_G, \on{Aff}_G) \longrightarrow (\on{Sph}_{\Gcheck}^{\on{spec}}, \on{Aff}_{\Gcheck}^{\on{spec}})
$$
in $\FactModCat$.

\sssec{}
There is a monoidal action
$$
    (\on{Sph}_G, \on{Aff}_G) \otimes (\on{Whit}^!(G), \on{Whit}^!(\on{Fl}_G)) \to (\on{Whit}^!(G), \on{Whit}^!(\on{Fl}_G))
$$
which can be written as a monoidal functor
$$
    (\on{Sph}_G, \on{Aff}_G) \to (\on{Whit}^!(G), \on{Whit}^!(\on{Fl}_G)) \otimes (\on{Whit}^*(G), \on{Whit}^*(\on{Fl}_G)) \simeq (\on{Rep}(\Gcheck), \operatorname{QCoh}(\ncheck/\Bcheck))^{\otimes 2}.
$$

\sssec{}
The image of the factorization (and monoidal) unit in $\on{Sph}_G$ under
$$
\on{Sph}_G \to \on{Rep}(\Gcheck)^{\otimes 2}
$$
is the factorization algebra $R_{\Gcheck}$, the regular representation of $\Gcheck$ considered as a representation of the group $\Gcheck \times \Gcheck$. Geometrically, this is the pushforward of
$$
\mathcal{O}_{\on{LS}_{\Gcheck}^{\on{reg}}}
$$
under
$$
\on{LS}_{\Gcheck}^{\on{reg}} \to \on{LS}_{\Gcheck}^{\on{reg}} \times \on{LS}_{\Gcheck}^{\on{reg}}.
$$

\sssec{}
Thus, we obtain a monoidal morphism of pairs 
$$
(\on{Sph}_G, \on{Aff}_G) \to R_{\Gcheck}\on{-FactMod}[(\on{Rep}(\Gcheck), \operatorname{QCoh}(\ncheck/\Bcheck))^{\otimes 2}].
$$
By Theorem \ref{approximation} (taking $\mathcal{Y}_0 = \on{pt}/\Gcheck^2 \times X_{dR}$, $\mathcal{Y} = \on{pt}/\Gcheck \times X_{dR}$, and $\mathcal{S}_0 = (\ncheck/\Bcheck_{\Ran_{x_0}})^2$), there is a morphism
$$
(\on{Sph}_{\Gcheck}^{\on{spec}}, \on{Aff}_{\Gcheck}^{\on{spec}}) \to R_{\Gcheck}\on{-FactMod}[(\on{Rep}(\Gcheck), \operatorname{QCoh}(\ncheck/\Bcheck))^{\otimes 2}]
$$
that is an equivalence on the eventually coconnective part of each component.

\sssec{}
As when we were constructing $\on{AB}$, we would like to say that $\on{Aff}_{G}$ is compactly generated and that its compact generators are sent to eventually coconnective objects. However, it is unlikely that $\on{Aff}_{G}$ is compactly generated. Again, we are saved by fusability, since 
$$
\on{Aff}_{G, {\on{Sph}_{\Gcheck}^{\on{spec}}\on{-QLisse}}} \subset \on{Aff}_G
$$
\emph{is} compactly generated by compact generators that are sent to
$$
R_{\Gcheck}\on{-FactMod}(\operatorname{QCoh}(\ncheck/\Bcheck)^{\otimes 2})^{> -\infty}.
$$
Using fusability of $\on{FactRes}_{\on{Rep}(\Gcheck) \to \on{Sph}_{\Gcheck}^{\on{spec}}}(\on{Aff}_{G})$, we are able to lift the resulting functor
$$
\on{Aff}_{G, {\on{Sph}_{\Gcheck}^{\on{spec}}\on{-QLisse}}} \longrightarrow \on{Aff}_{\Gcheck}^{\on{spec}}
$$
to all of $\on{Aff}_{G}$, and we obtain our desired morphism
$$
\on{B} : (\on{Sph}_G, \on{Aff}_{G}) \longrightarrow (\on{Sph}_{\Gcheck}^{\on{spec}}, \on{Aff}_{\Gcheck}^{\on{spec}})
$$
in $\FactModCat$.

\subsection{B is a Pointwise Equivalence} \label{B is a Pointwise Equivalence}

\sssec{}
After applying the appropriate factorization restrictions, $\on{Aff}_G$ and $\on{Aff}_{\Gcheck}^{\on{spec}}$ are fusable as factorization $\on{Rep}(\Gcheck)$-module categories. Thus, in order to show that $\on{B}$ is an equivalence, it suffices to show that 
$$
\on{B}_{x_0} : \on{Aff}_{G, x_0} \longrightarrow \on{Aff}_{\Gcheck, x_0}^{\on{spec}}
$$ is an equivalence of categories.

\sssec{}
In \cite{DT}, Dhillon and Taylor construct a version of (pointwise) Bezrukavnikov's equivalence in the monodromic, Betti setting. In other words, there are two key differences from our setting. First, the stack on the geometric side is
$$
\tilde{\on{Fl}}_G/I_0 = I_0 \backslash LG / I_0
$$
instead of
$$
\on{Fl}_G/I = I \backslash LG/I.
$$
Second, the sheaf theory on the geometric side is ind-constructible sheaves with nilpotent singular support
$$
\on{Shv}_{\on{Nilp}}
$$
instead of
$$
\Dmod
$$
and correspondingly, the spaces that appear on the spectral side are the \emph{Betti} versions.

\sssec{}
Let us summarize the construction in \cite{DT} (or, rather, the \emph{analogue} of their construction in our unipotent monodromy, de Rham setting). The monoidal category
$$
\Dmod(\on{Fl}_{G})^{(I, \on{ren})}
$$
has an obvious action on the category
$$
\on{Whit}^!(\on{Fl}_G)
$$
which can be identified via $\on{AB}_{x_0}$ with
$$
\QCoh(\tilde{\check{\mathcal{N}}}/\Gcheck).
$$
Furthermore, the action of $\Dmod(\on{Fl}_{G})^{(I, \on{ren})}$ on $\on{Whit}^!(\on{Fl}_G)$ has a tautological datum of compatibility with the action of
$$
\on{bi-Whit}(LG) = \on{Whit}^*(\on{Whit}^!(LG))
$$
on $\on{Whit}^!(\on{Fl}_G)$. However, the latter action can be identified with the action of
$$
\QCoh(\gcheck/\Gcheck)
$$
on $\QCoh(\tilde{\check{\mathcal{N}}}/\Gcheck)$ via pullback. Thus, we obtain a monoidal functor
$$
\iota^! : \Dmod(\on{Fl}_{G})^{(I, \on{ren})} \longrightarrow \on{End}_{\QCoh(\gcheck/\Gcheck)}(\QCoh(\tilde{\check{\mathcal{N}}}/\Gcheck)) \simeq \QCoh(\tilde{\check{\mathcal{N}}}/\Gcheck \underset{\gcheck/\Gcheck}{\times} \tilde{\check{\mathcal{N}}}/\Gcheck).
$$

\sssec{}
Then, because of Morita-theoretic properties of $\QCoh(\tilde{\check{\mathcal{N}}}/\Gcheck)$ as a $\Dmod(\on{Fl}_{G})^{(I, \on{ren})}$-$\QCoh(\gcheck/\Gcheck)$ bimodule, $\iota^!$ has a fully faithful left adjoint $\iota_!$. Furthermore, the kernel of $\iota^!$ consists precisely of the infinitely connective objects of $\Dmod(\on{Fl}_{G})^{(I, \on{ren})}$, so we obtain an equivalence
$$
\Dmod(\on{Fl}_{G})^{(I, \on{ren})}/\Dmod(\on{Fl}_{G})^{(I, \on{ren}), \leq -\infty} \overset{\simeq}{\longrightarrow} \QCoh(\tilde{\check{\mathcal{N}}}/\Gcheck \underset{\gcheck/\Gcheck}{\times} \tilde{\check{\mathcal{N}}}/\Gcheck).
$$

\sssec{}
Finally, we renormalize as follows. On the spectral side, the full subcategory of pseudocompact objects is
$$
\on{Psc}(\QCoh(\tilde{\check{\mathcal{N}}}/\Gcheck \underset{\gcheck/\Gcheck}{\times} \tilde{\check{\mathcal{N}}}/\Gcheck)) = \on{Coh}(\tilde{\check{\mathcal{N}}}/\Gcheck \underset{\gcheck/\Gcheck}{\times} \tilde{\check{\mathcal{N}}}/\Gcheck).
$$
On the other hand, on the geometric side we have an equivalence
$$
\on{Psc}(\Dmod(\on{Fl}_{G})^{(I, \on{ren})}) \simeq \on{Psc}(\Dmod(\on{Fl}_{G})^{(I, \on{ren})}/\Dmod(\on{Fl}_{G})^{(I, \on{ren}), \leq -\infty})
$$
and the pseudocompact objects in $\Dmod(\on{Fl}_{G})^{(I, \on{ren})}$ are precisely the compact objects (i.e. objects of $\Dmod(\on{Fl}_{G})^{I}$ that are compact after forgetting down to $\Dmod(\on{Fl}_G)$ on either side). Thus, we obtain an equivalence
$$
\on{B}^{\on{classical}} : \Dmod(\on{Fl}_{G})^{(I, \on{ren})} \overset{\simeq}{\longrightarrow} \IndCoh(\tilde{\check{\mathcal{N}}}/\Gcheck \underset{\gcheck/\Gcheck}{\times} \tilde{\check{\mathcal{N}}}/\Gcheck).
$$
%which is essentially (because $\on{AB}_{x_0}$ and $\on{AB}^{\on{classical}}$ differ by an automorphism) a modern reformulation of the equivalence constructed by Bezrukavnikov in \cite{B}. 

\sssec{}
Our task now is to identify $\on{B}_{x_0}$, the fiber at $x_0$ of our factorization functor $\on{B}$, with $\on{B}^{\on{classical}}$. To do so, it suffices to show that 
$$
\on{Aff}_{G, x_0} \overset{\on{B}_{x_0}}{\longrightarrow} \on{Aff}_{\Gcheck, x_0}^{\on{spec}} \overset{\Psi_{x_0}}{\longrightarrow} \QCoh_{\on{co}}(\ncheck/\Bcheck_{\Ran_{x_0}} \underset{(\on{LS}_{\Gcheck}^{\on{mer}})_{\Ran_{x_0}}}{\times} \ncheck/\Bcheck_{\Ran_{x_0}})_{x_0}
$$
identifies with
$$
\iota^!: \Dmod(\on{Fl}_{G})^{(I, \on{ren})} \overset{\on{B}^{\on{classical}}}{\longrightarrow} \IndCoh(\tilde{\check{\mathcal{N}}}/\Gcheck \underset{\gcheck/\Gcheck}{\times} \tilde{\check{\mathcal{N}}}/\Gcheck) \overset{\Psi}{\longrightarrow} \QCoh(\tilde{\check{\mathcal{N}}}/\Gcheck \underset{\gcheck/\Gcheck}{\times} \tilde{\check{\mathcal{N}}}/\Gcheck).
$$

\sssec{}
By construction, we already have an identification of
\begin{align}
\on{Aff}_{G, x_0} \overset{\on{B}_{x_0}}{\longrightarrow} \on{Aff}_{\Gcheck, x_0}^{\on{spec}} &\overset{\Psi_{x_0}}{\longrightarrow} \QCoh_{\on{co}}(\ncheck/\Bcheck_{\Ran_{x_0}} \underset{(\on{LS}_{\Gcheck}^{\on{mer}})_{\Ran_{x_0}}}{\times} \ncheck/\Bcheck_{\Ran_{x_0}})_{x_0} \\ &\longrightarrow \on{QCoh}(\ncheck/\Bcheck_{\Ran_{x_0}} \times \ncheck/\Bcheck_{\Ran_{x_0}})_{x_0}
\end{align}
with
\begin{align}
    \Dmod(\on{Fl}_{G})^{(I, \on{ren})} \overset{\on{B}^{\on{classical}}}{\longrightarrow} \IndCoh(\tilde{\check{\mathcal{N}}}/\Gcheck \underset{\gcheck/\Gcheck}{\times} \tilde{\check{\mathcal{N}}}/\Gcheck) &\overset{\Psi}{\longrightarrow} \QCoh(\tilde{\check{\mathcal{N}}}/\Gcheck \underset{\gcheck/\Gcheck}{\times} \tilde{\check{\mathcal{N}}}/\Gcheck) \\ &\longrightarrow \on{QCoh}(\tilde{\check{\mathcal{N}}}/\Gcheck \times \tilde{\check{\mathcal{N}}}/\Gcheck)
\end{align}
(i.e. the functors to $\on{End}(\on{QCoh}(\tilde{\check{\mathcal{N}}}/\Gcheck))$ defined by the two actions are the same because we used $\on{AB}_{x_0}$ to construct the latter action).

\sssec{}
Therefore, all that remains to be shown is that the datum of compatibility - between the action of $\on{Aff}_{G, x_0}$ on $\on{QCoh}(\tilde{\check{\mathcal{N}}}/\Gcheck)$ and the action of $\QCoh(\gcheck/\Gcheck)$ - given by $\Psi_{x_0} \circ \on{B}_{x_0}$ is isomorphic to the tautological one. But this follows from the fact that the pair
$$
(\on{Rep}(\Gcheck), \QCoh(\ncheck/\Bcheck))
$$
carries a monoidal action of the pair
$$
(\QCoh(\on{LS}_{\Gcheck}^{\on{mer}}), \QCoh(\on{LS}_{\Gcheck}^{\on{mer}}))
$$
which has a tautological datum of compatibility with the monoidal action of $(\on{Sph}_G, \on{Aff}_{G}$); and when restricted to $x_0$, the former action identifies with the action of
$$
\QCoh(\on{LS}_{\Gcheck}(D_{x_0}^{\times}))
$$
on
$$
\QCoh(\tilde{\check{\mathcal{N}}}/\Gcheck)
$$
given by pullback along
$$
\tilde{\check{\mathcal{N}}}/\Gcheck \longrightarrow\on{LS}_{\Gcheck}(D_{x_0}^{\times})
$$
which factors through $\gcheck/\Gcheck$. This concludes the proof of Theorem \ref{Theorem B}.

\section{Iwahori FLE} \label{Section IFLE}
\subsection{Construction of IFLE} \label{Construction of IFLE}
In this subsection, we will construct the Iwahori Fundamental Local Equivalence
$$
\on{IFLE}: (\KL, \IKL) \longrightarrow (\IndCoh^*(\on{Op}_{\Gcheck}^{\on{mon-free}}), \IndCoh^*(\on{Op}_{\Gcheck}^{\on{mer}} \underset{\on{LS}_{\Gcheck}^{\on{mer}}}{\times} \ncheck/\Bcheck)).
$$
as a morphism in $\FactModCat$.

\sssec{}
We recall that there is an action
$$
\on{act} : \on{Rep}(\Gcheck) \otimes \KL \to \KL
$$
obtained from the action of $\on{Sph}_{G}$ on $\KL$ via the monoidal functor
$$
\on{Sat}_G^{\on{nv}} : \on{Rep}(\Gcheck) \to \on{Sph}_{\Gcheck}^{\on{spec}} \simeq \on{Sph}_G.
$$

We have
$$
 \on{act}(R_{\Gcheck, \on{Op}} \boxtimes Vac) \simeq Vac
$$
so
$$
R_{\Gcheck, \on{Op}} \boxtimes Vac \to \on{coact}(Vac).
$$

\sssec{}
We can extend the action $\on{act}$ to
$$
(\on{Rep}(\Gcheck), \QCoh(\ncheck/\Bcheck)) \otimes (\KL, \IKL) \to (\KL, \IKL)
$$
by composing the action of $(\on{Sph}_{G}, \on{Aff}_G)$ on $(\KL, \IKL)$ with the monoidal morphism of pairs
$$
    \on{B}^{\on{nv}} : (\on{Rep}(\Gcheck), \QCoh(\ncheck/\Bcheck)) \to (\on{Sph}_{\Gcheck}^{\on{spec}}, \on{Aff}_{\Gcheck}^{\on{spec}}) \simeq (\on{Sph}_G, \on{Aff}_G).
$$

Now analogously with Section 5.3 of \cite{GLC2}, we consider the right adjoint of the above action
$$
(\KL, \IKL) \to (\on{Rep}(\Gcheck), \QCoh(\ncheck/\Bcheck)) \otimes (\KL, \IKL).
$$

This factors through
$$
(\KL, \IKL) \to \on{coact}(Vac)\on{-FactMod}[(\on{Rep}(\Gcheck), \QCoh(\ncheck/\Bcheck)) \otimes (\KL, \IKL)]
$$
and factorization restriction along $R_{\Gcheck, \on{Op}} \boxtimes Vac \to \on{coact}(Vac)$ gives a map to
$$
(R_{\Gcheck, \on{Op}} \boxtimes Vac)\on{-FactMod}[(\on{Rep}(\Gcheck), \QCoh(\ncheck/\Bcheck)) \otimes (\KL, \IKL)]
$$
which is the same as
$$
R_{\Gcheck, \on{Op}}\on{-FactMod}(\on{Rep}(\Gcheck), \QCoh(\ncheck/\Bcheck)) \otimes Vac\on{-FactMod}(\KL, \IKL)
$$
or in other words 
$$
R_{\Gcheck, \on{Op}}\on{-FactMod}(\on{Rep}(\Gcheck), \QCoh(\ncheck/\Bcheck)) \otimes (\KL, \IKL).
$$

However, by Theorem \ref{approximation}, there is a morphism of pairs
$$
(\IndCoh^*(\on{Op}_{\Gcheck}^{\on{mon-free}}), \IndCoh^*(\on{Op}_{\Gcheck}^{\on{mer}} \underset{\on{LS}_{\Gcheck}^{\on{mer}}}{\times} \ncheck/\Bcheck)) \to R_{\Gcheck, \on{Op}}\on{-FactMod}(\on{Rep}(\Gcheck), \QCoh(\ncheck/\Bcheck))
$$
that is an equivalence on the eventually coconnective part in each component.

\sssec{}
Then we are able to obtain a map from
$$
(\KL, \IKL)
$$
to
$$
(\IndCoh^*(\on{Op}_{\Gcheck}^{\on{mon-free}}), \IndCoh^*(\on{Op}_{\Gcheck}^{\on{mer}} \underset{\on{LS}_{\Gcheck}^{\on{mer}}}{\times} \ncheck/\Bcheck)) \otimes (\KL, \IKL)
$$
since on the first component, as noted in loc. cit.
$$
\KL \to \on{Rep}(\Gcheck) \otimes \KL
$$
is t-exact so
$$
\KL^c \subset \KL^b
$$
is sent to the eventually coconnective part of 
$$
R_{\Gcheck, \on{Op}}\on{-FactMod}(\on{Rep}(\Gcheck)) \otimes \KL.
$$
On the second component, 
$$
\on{Aff}_{\Gcheck}^{\on{spec}} \otimes \IKL \to \IKL
$$
is bounded on the right for each fixed bounded object of $\IKL$. Then so is
$$
\QCoh(\ncheck/\Bcheck) \otimes \IKL \to \IKL
$$
and thus the right adjoint
$$
\IKL \to \QCoh(\ncheck/\Bcheck) \otimes \IKL
$$
is bounded on the left for each fixed bounded object of $\IKL$. Therefore, 
$$
\IKL^c \subset \IKL^b
$$
is sent to the eventually coconnective part of 
$$
R_{\Gcheck, \on{Op}}\on{-FactMod}(\QCoh(\ncheck/\Bcheck)) \otimes \IKL.
$$

\sssec{}
This morphism defines a coaction of 
$$
(\IndCoh^*(\on{Op}_{\Gcheck}^{\on{mon-free}}), \IndCoh^*(\on{Op}_{\Gcheck}^{\on{mer}} \underset{\on{LS}_{\Gcheck}^{\on{mer}}}{\times} \ncheck/\Bcheck))
$$
on $(\KL, \IKL)$, or in other words, an action of 
$$
(\IndCoh^!(\on{Op}_{\Gcheck}^{\on{mon-free}}), \IndCoh^!(\on{Op}_{\Gcheck}^{\on{mer}} \underset{\on{LS}_{\Gcheck}^{\on{mer}}}{\times} \ncheck/\Bcheck)).
$$

\sssec{}
Parallel to Section 6.1 of \cite{GLC2}, we can consider the morphism of pairs
\begin{align}
(\KL, \IKL) &\to (\on{Whit}^!(\gmod_{\on{crit}}), \on{Whit}^!(\gmod_{\on{crit}}))  \\&\overset{\on{\overline{DS}^{enh, rfnd}}}{\to} (\IndCoh^*(\on{Op}_{\Gcheck}^{\on{mer}}), \IndCoh^*(\on{Op}_{\Gcheck}^{\on{mer}}))
\end{align}
as a morphism of $(\IndCoh^!(\on{Op}_{\Gcheck}^{\on{mer}}), \IndCoh^!(\on{Op}_{\Gcheck}^{\on{mer}}))$-modules (in the monoidal sense).

We thus obtain a morphism of pairs from
$$
(\KL, \IKL)
$$ 
to
$$
\on{Fun}_{\IndCoh^!(\on{Op}_{\Gcheck}^{\on{mer}})}((\IndCoh^!(\on{Op}_{\Gcheck}^{\on{mon-free}}), \IndCoh^!(\on{Op}_{\Gcheck}^{\on{mer}} \underset{\on{LS}_{\Gcheck}^{\on{mer}}}{\times} \ncheck/\Bcheck)), \IndCoh^*(\on{Op}_{\Gcheck}^{\on{mer}}))
$$
which identifies with
$$
(\IndCoh^*(\on{Op}_{\Gcheck}^{\on{mon-free}}), \IndCoh^*(\on{Op}_{\Gcheck}^{\on{mer}} \underset{\on{LS}_{\Gcheck}^{\on{mer}}}{\times} \ncheck/\Bcheck)).
$$
This morphism of pairs is what we will denote by $\on{IFLE}$.

\begin{rem}
$\on{IFLE}$ intertwines (more precisely, has intertwining \emph{data}) the action of 
$$
(\on{Sph}_{G}, \on{Aff}_G)
$$ on 
$$
(\KL, \IKL)
$$
and the action of 
$$
(\on{Sph}_{\Gcheck}^{\on{spec}}, \on{Aff}_{\Gcheck}^{\on{spec}})
$$ on 
$$
(\IndCoh^*(\on{Op}_{\Gcheck}^{\on{mon-free}}), \IndCoh^*(\on{Op}_{\Gcheck}^{\on{mer}} \underset{\on{LS}_{\Gcheck}^{\on{mer}}}{\times} \ncheck/\Bcheck))
$$
via the monoidal equivalence of algebra objects in $\FactModCat$ 
$$
\on{B} : (\on{Sph}_{G}, \on{Aff}_G) \simeq (\on{Sph}_{\Gcheck}^{\on{spec}}, \on{Aff}_{\Gcheck}^{\on{spec}}).
$$
The reason is that everything from Appendix E.10.4 to the end of E.11 in \cite{GLC2} carries over \emph{mutatis mutandis} to the setting of pairs (and Lemma E.10.3 is replaced by our construction of $\on{B}$).
\end{rem}

\subsection{Reduction to the Pointwise Equivalence} \label{Reduction to the Pointwise Equivalence}
We will now show that to prove $\on{IFLE}$ is an equivalence of pairs, it is sufficient to prove that it is an equivalence at $x_0$. Our strategy is similar to that for $\on{FLE}$ in Section 6.2 of \cite{GLC2}.

\sssec{}
For the rest of the subsection, let us assume the following
\begin{prop} \label{pointwise IFLE}
$$
\on{IFLE}_{x_0} : \IKLx \to \IndCoh^*(\on{Op}_{\Gcheck}^{\on{mer}} \underset{\on{LS}_{\Gcheck}^{\on{mer}}}{\times} \ncheck/\Bcheck)_{x_0}
$$
is an equivalence.
\end{prop}
This proposition will be proven in the latter subsections of this section.

\sssec{}
We note that on each stratum $X_{x_0}^{(n), \circ}$ of $\Ran_{x_0}$,
$$
\on{IFLE}_{X_{x_0}^{(n), \circ}} : (\IKL)_{X_{x_0}^{(n), \circ}} \to \IndCoh^*(\on{Op}_{\Gcheck}^{\on{mer}} \underset{\on{LS}_{\Gcheck}^{\on{mer}}}{\times} \ncheck/\Bcheck)_{X_{x_0}^{(n), \circ}}
$$
identifies with
\begin{align}
\on{FLE}_{X^{(n), \circ}} \boxtimes \on{IFLE}_{x_0} : &(\KL)_{X^{(n), \circ}} \boxtimes \IKLx \to \\
&\IndCoh^*(\on{Op}_{\Gcheck}^{\on{mon-free}})_{X^{(n), \circ}} \boxtimes \IndCoh^*(\on{Op}_{\Gcheck}^{\on{mer}} \underset{\on{LS}_{\Gcheck}^{\on{mer}}}{\times} \ncheck/\Bcheck)_{x_0}
\end{align}
which is an equivalence because $\on{FLE}$ is an equivalence.

\sssec{}
Now it suffices to show that $\IKL$ is compactly generated and that $\on{IFLE}$ preserves compacts. This will demonstrate that $\on{IFLE}$ has a continuous right adjoint as a morphism of sheaves of categories, which implies that whether it is an equivalence can be checked strata by strata.

\sssec{}
In the same way that $\KL$ is generated by the images of the compact generators of $\on{Rep}(L^+G)$ under
$$
\on{ind}_{L^+G}^{(\hat{g}_{\on{crit}}, L^+G)} : \on{Rep}(L^+G) \to \KL,
$$
the category $\IKL$ is generated by the images of the compact generators of the factorization $\on{Rep}(L^+G)$-module category $\on{Rep}(I)$ under
$$
\on{ind}_{I}^{(\hat{g}_{\on{crit}}, I)} : \on{Rep}(I) \to \IKL
$$
which can be verified by mimicking the construction from Appendix B.14 of \cite{GLC2}.

\begin{rem}
    The argument below was explained to us by Lin Chen.
\end{rem}
\begin{lem}
    An object $c \in \IndCoh^*(\on{Op}_{\Gcheck}^{\on{mer}} \underset{\on{LS}_{\Gcheck}^{\on{mer}}}{\times} \ncheck/\Bcheck)$ is compact if for every $\lambda \in \Lambda_{G}$, the *-pushforward along
    $$
    \on{Op}_{\Gcheck}^{\on{mer}} \underset{\on{LS}_{\Gcheck}^{\on{mer}}}{\times} \ncheck/\Bcheck \longrightarrow \on{Op}_{\Gcheck}^{\on{mer}}
    $$
    of
    $$
    c(\lambda) = c \otimes (\on{Op}_{\Gcheck}^{\on{mer}} \underset{\on{LS}_{\Gcheck}^{\on{mer}}}{\times} \ncheck/\Bcheck \to \ncheck/\Bcheck)^! \mathcal{O}(\lambda)
    $$
    is compact in $\IndCoh^*(\on{Op}_{\Gcheck}^{\on{mer}})$.
\end{lem}

\begin{proof}
    The map 
    $$
    \on{Op}_{\Gcheck}^{\on{mer}} \underset{\on{LS}_{\Gcheck}^{\on{mer}}}{\times} \ncheck/\Bcheck \longrightarrow \on{Op}_{\Gcheck}^{\on{mer}}
    $$
    is a base change of
    $$
    \tilde{\check{\mathcal{N}}}/\Gcheck \longrightarrow \check{\mathcal{N}}/\Gcheck \longrightarrow \on{LS}_{\Gcheck}^{\on{mer}}.
    $$
    The second map is a closed embedding, so it detects coherent objects. Thus, it suffices to show the claim for 
    $$
    \tilde{\check{\mathcal{N}}}/\Gcheck \longrightarrow \check{\mathcal{N}}/\Gcheck
    $$
    and by fppf descent, we can reduce to showing the claim for the springer resolution
    $$
    \tilde{\check{\mathcal{N}}} \longrightarrow \check{\mathcal{N}}
    $$
    which factors as
    $$
    \tilde{\check{\mathcal{N}}} \longrightarrow \check{\mathcal{N}} \times \Gcheck/\Bcheck \longrightarrow \check{\mathcal{N}}.
    $$
    Since the first map is a closed embedding, it detects coherent objects, and it suffices to show the claim for
    $$
    \Gcheck/\Bcheck \longrightarrow \on{pt}.
    $$
    But on $\Gcheck/\Bcheck$, the claim that the $\lambda$-twisted global sections functors jointly detect coherent objects follows from the fact that $\QCoh(\Gcheck/\Bcheck) = \IndCoh(\Gcheck/\Bcheck)$ is generated by a finite number of $\mathcal{O}(\lambda)$, and thus
    $$
    \QCoh(\Gcheck/\Bcheck) \simeq R\on{-mod}
    $$
    for a smooth proper dg algebra $R$, where the forgetful functor $R\on{-mod} \to \on{Vect}$ corresponds to a finite sum of twisted global sections functors. But this clearly detects coherent objects.
\end{proof}

\begin{prop}
The functor $\on{IFLE}$ preserves compact objects.
\end{prop}

\begin{proof}
    $\IKL$ is compactly generated by the verma modules $\mathbb{M}^{\check{\lambda}}$, so we want to show that $\on{IFLE}(\mathbb{M}^{\check{\lambda}})$ is compact. By the above lemma, it suffices to show that
    $$
    (\on{Op}_{\Gcheck}^{\on{mer}} \underset{\on{LS}_{\Gcheck}^{\on{mer}}}{\times} \ncheck/\Bcheck \to \on{Op}_{\Gcheck}^{\on{mer}})_*(\on{IFLE}(\mathbb{M}^{\check{\lambda}})(\lambda))
    $$
    is compact. However, $\on{IFLE}$ is compatible with $\on{B}$, so twisting by $\lambda$ on the spectral side corresponds to convolving with the Wakimoto $J_{\lambda}$ on the geometric side, i.e.
    $$
    \on{IFLE}(\mathbb{M}^{\check{\lambda}})(\lambda) = \on{IFLE}(\mathbb{M}^{\check{\lambda}} \star J_{\lambda}).
    $$
    Furthermore, 
    $$
    (\on{Op}_{\Gcheck}^{\on{mer}} \underset{\on{LS}_{\Gcheck}^{\on{mer}}}{\times} \ncheck/\Bcheck \to \on{Op}_{\Gcheck}^{\on{mer}})_* \circ \on{IFLE} : \IKL \longrightarrow \IndCoh^*(\on{Op}_{\Gcheck}^{\on{mer}})
    $$
    corresponds to the functor
    $$
    \mathcal{C}^I \longrightarrow \mathcal{C} \longrightarrow \on{Whit}^*(\mathcal{C})
    $$
    for $\mathcal{C} = \gmod_{\on{crit}}$, which preserves compacts. $\mathbb{M}^{\check{\lambda}}$ is compact, and since convolving with a Wakimoto is an equivalence, so is $\mathbb{M}^{\check{\lambda}} \star J_{\lambda}$, and we are done.
\end{proof}

\subsection{Iwahori-Hecke Temperedness} \label{Affine Temperedness}
In this subsection, we will introduce the notion of \emph{Iwahori-Hecke Temperedness} and explain how to use it to prove Proposition \ref{pointwise IFLE}.

\sssec{}
Consider the fiber of $\on{Aff}_{\Gcheck}^{\on{spec}}$ at $x_0$:
$$
\IndCoh^*(\ncheck/\Bcheck_{\Ran_{x_0}} \underset{(\on{LS}_{\Gcheck}^{\on{mer}})_{\Ran_{x_0}}}{\times} \ncheck/\Bcheck_{\Ran_{x_0}})_{x_0} \simeq \IndCoh(\ncheck/\Bcheck \underset{\gcheck/\Gcheck}{\times} \ncheck/\Bcheck).
$$
Denote by $\on{Aff}_{\Gcheck, x_0}^{\on{spec, temp}}$ the category
$$
\QCoh(\ncheck/\Bcheck \underset{\gcheck/\Gcheck}{\times} \ncheck/\Bcheck).
$$
There is a fully faithful embedding
$$
\on{Aff}_{\Gcheck, x_0}^{\on{spec, temp}} \to \on{Aff}_{\Gcheck, x_0}^{\on{spec}}
$$
which has a monoidal continuous right adjoint
$$
\on{Aff}_{\Gcheck, x_0}^{\on{spec}} \to \on{Aff}_{\Gcheck, x_0}^{\on{spec, temp}}.
$$

\sssec{}
Let $\boldsymbol{M}$ be a $\on{Aff}_{\Gcheck, x_0}^{\on{spec}}$-module category. Then we can consider
$$
\boldsymbol{M}^{\on{temp}} = \on{Aff}_{\Gcheck, x_0}^{\on{spec, temp}} \underset{\on{Aff}_{\Gcheck, x_0}^{\on{spec}}}{\otimes} \boldsymbol{M}
$$
and the adjunction
$$
\on{Aff}_{\Gcheck, x_0}^{\on{spec, temp}} \rightleftharpoons \on{Aff}_{\Gcheck, x_0}^{\on{spec}}
$$
induces an adjunction
$$
\boldsymbol{M}^{\on{temp}} \rightleftharpoons \boldsymbol{M}.
$$

\begin{defn}
    We say that a $\on{Aff}_{\Gcheck, x_0}^{\on{spec}}$-module category $\boldsymbol{M}$ is \emph{Iwahori-Hecke tempered} if the above adjunction is an equivalence.
\end{defn}

\begin{rem}
    We use the adverb \emph{Iwahori-Hecke} in order to distinguish our notion from the analogous notion of temperedness defined in Section 7.1 of \cite{GLC2}. We will refer to the latter as \emph{spherical} temperedness.
\end{rem}

\sssec{}
Now let
$$
\on{LS}_{\Gcheck, x_0}^{\on{nilp-mon}}
$$
be the image of
$$
\ncheck/\Bcheck \to \on{LS}_{\Gcheck, x_0}^{\on{mer}}
$$
and note that
$$
\on{Aff}_{\Gcheck, x_0}^{\on{spec, temp}} \simeq \on{Fun}_{\QCoh((\on{LS}_{\Gcheck, x_0}^{\on{mer}})_{\on{nilp-mon}}^\wedge)}(\QCoh((\ncheck/\Bcheck)_{x_0}), \QCoh((\ncheck/\Bcheck)_{x_0})).
$$
For a $\on{Aff}_{\Gcheck, x_0}^{\on{spec}}$-module category $\boldsymbol{M}$, since $\QCoh((\ncheck/\Bcheck)_{x_0})$ is rigid, entirely parallel to Corollary 7.1.8 in \cite{GLC2}, we can conclude that
$$
\boldsymbol{M}^{\on{temp}} \simeq \on{Fun}_{\QCoh((\on{LS}_{\Gcheck, x_0}^{\on{mer}})_{\on{nilp-mon}}^\wedge)}(\QCoh((\ncheck/\Bcheck)_{x_0}), \QCoh((\ncheck/\Bcheck)_{x_0}) \underset{\on{Aff}_{\Gcheck, x_0}^{\on{spec}}}{\otimes} \boldsymbol{M})
$$
and thus on the full subcategory of Iwahori-Hecke tempered $\on{Aff}_{\Gcheck, x_0}^{\on{spec}}$-module categories, the assignment
$$
\boldsymbol{M} \mapsto \QCoh((\ncheck/\Bcheck)_{x_0}) \underset{\on{Aff}_{\Gcheck, x_0}^{\on{spec}}}{\otimes} \boldsymbol{M}
$$
is conservative.

\sssec{}
Here is our strategy for proving Proposition \ref{pointwise IFLE}. It is similar to the argument given throughout Section 7.2 to 7.4 of \cite{GLC2}. First, we will construct a commutative diagram
$$
\begin{tikzcd}
	{\on{Whit}^*(\on{Fl}_G)_{x_0} \underset{\on{Aff}_{G, x_0}}{\otimes} \IKLx} && {\QCoh(\ncheck/\Bcheck)_{x_0} \underset{\on{Aff}_{G, x_0}^{\on{spec}}}{\otimes} \IndCoh^*(\on{Op}_{\Gcheck}^{\on{mer}} \underset{\on{LS}_{\Gcheck}^{\on{mer}}}{\times} \ncheck/\Bcheck)_{x_0}} \\
	\\
	{\IndCoh(\on{Op}_{\Gcheck, x_0}^{\on{mer}})} & {\boldsymbol{=}} & {\IndCoh(\on{Op}_{\Gcheck, x_0}^{\on{mer}}).}
	\arrow[from=1-1, to=1-3]
	\arrow[from=1-1, to=3-1]
	\arrow[from=1-3, to=3-3]
\end{tikzcd}
$$
Then, we will show that the vertical arrows are fully faithful embeddings with the same essential image, which will demonstrate that the top arrow
\begin{align}
\on{Whit}^*(\on{Fl}_G)_{x_0} \underset{\on{Aff}_{G, x_0}}{\otimes} \on{IKL}(G)_{\on{crit}, x_0} &\longrightarrow \QCoh(\ncheck/\Bcheck)_{x_0} \underset{\on{Aff}_{G, x_0}^{\on{spec}}}{\otimes} \on{IKL}(G)_{\on{crit}, x_0} \\ &\overset{\on{Id} \otimes \on{IFLE}_{x_0}}{\longrightarrow} \QCoh(\ncheck/\Bcheck)_{x_0} \underset{\on{Aff}_{G, x_0}^{\on{spec}}}{\otimes} \IndCoh^*(\on{Op}_{\Gcheck}^{\on{mer}} \underset{\on{LS}_{\Gcheck}^{\on{mer}}}{\times} \ncheck/\Bcheck)_{x_0}
\end{align}
is an equivalence. Since the first morphism of the composition is an equivalence, this will imply that the second morphism of the composition is also an equivalence.

Finally, we will show that $\IKLx$ and 
$$
\IndCoh^*(\on{Op}_{\Gcheck}^{\on{mer}} \underset{\on{LS}_{\Gcheck}^{\on{mer}}}{\times} \ncheck/\Bcheck)_{x_0} \simeq \IndCoh(\on{Op}_{\Gcheck, x_0}^{\on{mer}} \underset{\on{LS}_{\Gcheck, x_0}^{\on{mer}}}{\times} (\ncheck/\Bcheck)_{x_0})
$$
are both Iwahori-Hecke tempered, which will allow us to conclude that $\on{IFLE}_{x_0}$ is an equivalence.

\subsection{IFLE is a Pointwise Equivalence} \label{IFLE is a Pointwise Equivalence}
In this subsection, we will check the various claims in order to prove Proposition \ref{pointwise IFLE}.

\begin{prop}
    There is a commutative diagram
$$
\begin{tikzcd}
	{\on{Whit}^*(\on{Fl}_G)_{x_0} \underset{\on{Aff}_{G, x_0}}{\otimes} \IKLx} && {\QCoh(\ncheck/\Bcheck)_{x_0} \underset{\on{Aff}_{G, x_0}^{\on{spec}}}{\otimes} \IndCoh^*(\on{Op}_{\Gcheck}^{\on{mer}} \underset{\on{LS}_{\Gcheck}^{\on{mer}}}{\times} \ncheck/\Bcheck)_{x_0}} \\
	\\
	{\IndCoh(\on{Op}_{\Gcheck, x_0}^{\on{mer}})} & {\boldsymbol{=}} & {\IndCoh(\on{Op}_{\Gcheck, x_0}^{\on{mer}}).}
	\arrow[from=1-1, to=1-3]
	\arrow[from=1-1, to=3-1]
	\arrow[from=1-3, to=3-3]
\end{tikzcd}
$$
\end{prop}
\begin{proof}
    We claim that there is a commutative diagram in $\FactModCat$
$$
    \begin{tikzcd}
	{\IKL} && {\IndCoh^*(\on{Op}_{\Gcheck}^{\on{mer}} \underset{\on{LS}_{\Gcheck}^{\on{mer}}}{\times} \ncheck/\Bcheck)} \\
	\\
	{\on{Whit}^!(\on{Fl}_G) \otimes \IndCoh^*(\on{Op}_{\Gcheck, x_0}^{\on{mer}})} && {\QCoh(\ncheck/\Bcheck) \otimes \IndCoh^*(\on{Op}_{\Gcheck, x_0}^{\on{mer}}).}
	\arrow[from=1-1, to=1-3]
	\arrow[from=1-1, to=3-1]
	\arrow[from=1-3, to=3-3]
    \arrow[from=3-1, to=3-3]
\end{tikzcd}
$$
where the right vertical arrow is given by $*$-pushforward along
$$
\on{Op}_{\Gcheck}^{\on{mer}} \underset{\on{LS}_{\Gcheck}^{\on{mer}}}{\times} \ncheck/\Bcheck \longrightarrow \on{Op}_{\Gcheck}^{\on{mer}} \times \ncheck/\Bcheck.
$$
For the left vertical arrow, consider the following paradigm: for $\mathcal{M}$ a category with an action of $\Dmodh(LG)$, there is a map
$$
    \Dmodh(\on{Fl}_G) \underset{\on{Aff}_G}{\otimes} \mathcal{M}^I \to \mathcal{M}
$$
taking Whittaker coinvariants yields
$$
\on{Whit}^*(\on{Fl}_G) \underset{\on{Aff}_G}{\otimes} \mathcal{M}^I \to \on{Whit}^*(\mathcal{M})
$$
and by duality we obtain
$$
\mathcal{M}^I \to \on{Whit}^!(\on{Fl}_G) \otimes \on{Whit}^*(\mathcal{M}).
$$
The left vertical arrow is this functor with $\mathcal{M} = \gmod_{\on{crit}}$.

The above commutative diagram can be constructed by following the construction in Appendix E.10 of \cite{GLC2} and upgrading it to pairs \emph{mutatis mutandis}.

Then we can invoke the compatibility of $\on{AB}$ and $\on{IFLE}$ with $\on{B}$ to curry and obtain
$$
\begin{tikzcd}
	{\on{Whit}^*(\on{Fl}_G) \underset{\on{Aff}_{G}}{\otimes} \IKL} && {\QCoh(\ncheck/\Bcheck) \underset{\on{Aff}_{G}^{\on{spec}}}{\otimes} \IndCoh^*(\on{Op}_{\Gcheck}^{\on{mer}} \underset{\on{LS}_{\Gcheck}^{\on{mer}}}{\times} \ncheck/\Bcheck)} \\
	\\
	{\IndCoh^*(\on{Op}_{\Gcheck}^{\on{mer}})} & {\boldsymbol{=}} & {\IndCoh^*(\on{Op}_{\Gcheck}^{\on{mer}}).}
	\arrow[from=1-1, to=1-3]
	\arrow[from=1-1, to=3-1]
	\arrow[from=1-3, to=3-3]
\end{tikzcd}
$$

The desired diagram is the specialization of the above to $x_0$.
\end{proof}

\begin{prop}
    The right vertical arrow
    $$
    \QCoh(\ncheck/\Bcheck)_{x_0} \underset{\on{Aff}_{G, x_0}^{\on{spec}}}{\otimes} \IndCoh(\on{Op}_{\Gcheck, x_0}^{\on{mer}} \underset{\on{LS}_{\Gcheck, x_0}^{\on{mer}}}{\times} (\ncheck/\Bcheck)_{x_0}) \longrightarrow 
    \IndCoh(\on{Op}_{\Gcheck, x_0}^{\on{mer}})
    $$
    is fully faithful with essential image
    $$
    \IndCoh(\on{Op}_{\Gcheck, x_0}^{\on{mer}})_{\on{Op}_{\Gcheck, x_0}^{\on{nilp-mon}}} \subset \IndCoh(\on{Op}_{\Gcheck, x_0}^{\on{mer}}),
    $$
    the full subcategory of $\IndCoh(\on{Op}_{\Gcheck, x_0}^{\on{mer}})$ of objects that are set-theoretically supported on
    $$
    \on{Op}_{\Gcheck, x_0}^{\on{nilp-mon}} = \on{Op}_{\Gcheck, x_0}^{\on{mer}} \underset{\on{LS}_{\Gcheck, x_0}^{\on{mer}}}{\times} \on{LS}_{\Gcheck, x_0}^{\on{nilp-mon}} \subset \on{Op}_{\Gcheck, x_0}^{\on{mer}}.
    $$
\end{prop}
\begin{proof}
    $\QCoh(\ncheck/\Bcheck)_{x_0}$ and $\on{Aff}_{\Gcheck, x_0}^{\on{spec}}$ are both rigid like their spherical counterparts, so we can apply the argument in the proof of Proposition 3.6.5 in \cite{GLC2} largely unchanged: the functor
    \begin{align}
    \IndCoh(\on{Op}_{\Gcheck, x_0}^{\on{mer}} \underset{\on{LS}_{\Gcheck, x_0}^{\on{mer}}}{\times} (\ncheck/\Bcheck)_{x_0}) &\simeq \QCoh(\ncheck/\Bcheck)_{x_0} \underset{\QCoh(\ncheck/\Bcheck)_{x_0}}{\otimes} \IndCoh(\on{Op}_{\Gcheck, x_0}^{\on{mer}} \underset{\on{LS}_{\Gcheck, x_0}^{\on{mer}}}{\times} (\ncheck/\Bcheck)_{x_0}) \\ &\to \QCoh(\ncheck/\Bcheck)_{x_0} \underset{\on{Aff}_{G, x_0}^{\on{spec}}}{\otimes} \IndCoh(\on{Op}_{\Gcheck, x_0}^{\on{mer}} \underset{\on{LS}_{\Gcheck, x_0}^{\on{mer}}}{\times} (\ncheck/\Bcheck)_{x_0})
    \end{align}
    has a continuous and monadic right adjoint. The same is true of the pushforward
    $$
    \IndCoh(\on{Op}_{\Gcheck, x_0}^{\on{mer}} \underset{\on{LS}_{\Gcheck, x_0}^{\on{mer}}}{\times} (\ncheck/\Bcheck)_{x_0}) \to \IndCoh(\on{Op}_{\Gcheck, x_0}^{\on{mer}})_{\on{Op}_{\Gcheck, x_0}^{\on{nilp-mon}}}
    $$
    because
    $$
    \ncheck/\Bcheck \to \on{LS_{\Gcheck, x_0}^{\on{nilp-mon}}}
    $$
    is the springer resolution, and thus proper and surjective. The functor
    $$
    \QCoh(\ncheck/\Bcheck)_{x_0} \underset{\on{Aff}_{G, x_0}^{\on{spec}}}{\otimes} \IndCoh(\on{Op}_{\Gcheck, x_0}^{\on{mer}} \underset{\on{LS}_{\Gcheck, x_0}^{\on{mer}}}{\times} (\ncheck/\Bcheck)_{x_0}) \longrightarrow 
    \IndCoh(\on{Op}_{\Gcheck, x_0}^{\on{mer}})
    $$
    induces a morphism between the two monads, and to see that it is an isomorphism of monads, it suffices to show that it is an isomorphism on the level of endofunctors. However, both monads are given by the action of the algebra object
    $$
    (\on{pr_1} : (\ncheck/\Bcheck)_{x_0} \underset{\on{LS}_{\Gcheck, x_0}^{\on{mer}}}{\times} (\ncheck/\Bcheck)_{x_0} \to (\ncheck/\Bcheck)_{x_0})^! \mathcal{O}_{(\ncheck/\Bcheck)_{x_0}}
    $$
    of $\on{Aff}_{\Gcheck, x_0}^{\on{spec}}$ on $\IndCoh(\on{Op}_{\Gcheck, x_0}^{\on{mer}} \underset{\on{LS}_{\Gcheck, x_0}^{\on{mer}}}{\times} (\ncheck/\Bcheck)_{x_0})$.
\end{proof}

\begin{prop}
    The left vertical arrow
    $$
    \on{Whit}^*(\on{Fl}_G)_{x_0} \underset{\on{Aff}_{G, x_0}}{\otimes} \IKLx \longrightarrow \IndCoh(\on{Op}_{\Gcheck, x_0}^{\on{mer}})
    $$
    is fully faithful with essential image
    $$
    \IndCoh(\on{Op}_{\Gcheck, x_0}^{\on{mer}})_{\on{Op}_{\Gcheck, x_0}^{\on{nilp-mon}}} \subset \IndCoh(\on{Op}_{\Gcheck, x_0}^{\on{mer}}).
    $$ 
\end{prop}
\begin{proof}
    Like the affine grassmannian, the affine flag variety $\on{Fl}_{\Gcheck, x_0}$ is also ind-proper. Thus, the proof of full faithfulness is similar to the argument given in Section 7.3 of \cite{GLC2}: just as in the spherical case, for a category $\boldsymbol{C}$ with an action of the loop group $LG$ at the critical level,
    $$
    \boldsymbol{C}^{\on{Aff-gen}} = \Dmodh(\on{Fl}_{G, x_0}) \underset{\on{Aff_{G, x_0}}}{\otimes} \boldsymbol{C}^I \longrightarrow \boldsymbol{C}
    $$
    is fully faithful and has a continuous right adjoint. But this implies that 
    $$
    \on{Whit}^*(\boldsymbol{C}^{\on{Aff-gen}}) = \on{Whit}^*(\on{Fl}_G)_{x_0} \underset{\on{Aff_{G, x_0}}}{\otimes} \boldsymbol{C}^I \longrightarrow \on{Whit}^*(\boldsymbol{C})
    $$
    is fully faithful. Taking $\boldsymbol{C} = \gmod_{\on{crit}}$ gives us fully faithfulness of
    $$
    \on{Whit}^*(\on{Fl}_G)_{x_0} \underset{\on{Aff}_{G, x_0}}{\otimes} \IKLx \longrightarrow \on{Whit}^*(\gmod_{\on{crit}}) \simeq \IndCoh(\on{Op}_{\Gcheck, x_0}^{\on{mer}})
    $$
    as desired.

    Furthermore, the functor
    $$
    \IKLx \overset{c \ \mapsto \ \boldsymbol{1}_{\on{Aff}_G} \otimes c}{\longrightarrow} \on{Whit}^*(\on{Fl}_G)_{x_0} \underset{\on{Aff}_{G, x_0}}{\otimes} \IKLx
    $$
    generates the target; therefore, to  show that the essential image of 
    $$
    \on{Whit}^*(\on{Fl}_G)_{x_0} \underset{\on{Aff}_{G, x_0}}{\otimes} \IKLx \to \IndCoh(\on{Op}_{\Gcheck, x_0}^{\on{mer}})
    $$
    is contained in 
    $$
    \IndCoh(\on{Op}_{\Gcheck, x_0}^{\on{mer}})_{\on{Op}_{\Gcheck, x_0}^{\on{nilp-mon}}},
    $$
    it suffices to show that the essential image of the composition
    $$
    \IKLx \to \on{Whit}^*(\on{Fl}_G)_{x_0} \underset{\on{Aff}_{G, x_0}}{\otimes} \IKLx \to \IndCoh(\on{Op}_{\Gcheck, x_0}^{\on{mer}})
    $$
    is contained in $\IndCoh(\on{Op}_{\Gcheck, x_0}^{\on{mer}})_{\on{Op}_{\Gcheck, x_0}^{\on{nilp-mon}}}$. However, the above composition is isomorphic to the composition
    $$
    \IKLx \overset{\on{IFLE}_{x_0}}{\longrightarrow} \IndCoh(\on{Op}_{\Gcheck, x_0}^{\on{mer}} \underset{\on{LS}_{\Gcheck, x_0}^{\on{mer}}}{\times} (\ncheck/\Bcheck)_{x_0}) \to 
    \IndCoh(\on{Op}_{\Gcheck, x_0}^{\on{mer}})
    $$
    whose essential image is clearly contained in $\IndCoh(\on{Op}_{\Gcheck, x_0}^{\on{mer}})_{\on{Op}_{\Gcheck, x_0}^{\on{nilp-mon}}}$.

    Now, it remains to show that 
    $$
    \on{Whit}^*(\on{Fl}_G)_{x_0} \underset{\on{Aff}_{G, x_0}}{\otimes} \IKLx \to \IndCoh(\on{Op}_{\Gcheck, x_0}^{\on{mer}})_{\on{Op}_{\Gcheck, x_0}^{\on{nilp-mon}}}
    $$
    is essentially surjective.
    
    Since the functor is fully faithful, it suffices to show that it generates the target; for which, in turn, it suffices to show that the composition
    $$
    \IKLx \to \on{Whit}^*(\on{Fl}_G)_{x_0} \underset{\on{Aff}_{G, x_0}}{\otimes} \IKLx \to \IndCoh(\on{Op}_{\Gcheck, x_0}^{\on{mer}})_{\on{Op}_{\Gcheck, x_0}^{\on{nilp-mon}}}
    $$
    generates the target.

    Because we have
    $$
    (\on{Op}_{\Gcheck, x_0}^{\on{nilp-mon}})_{\on{red}} = \bigsqcup_{\check{\lambda} + \check{\rho} \in \check{\lambda}^+} \on{Op}_{\Gcheck, x_0}^{\check{\lambda}-\on{nilp}},
    $$
    the objects 
    $$
    \mathcal{O}_{\on{Op}_{\Gcheck, x_0}^{\check{\lambda}-\on{nilp}}} \in \IndCoh(\on{Op}_{\Gcheck, x_0}^{\on{mer}})_{\on{Op}_{\Gcheck, x_0}^{\on{nilp-mon}}}
    $$
    compactly generate $\IndCoh(\on{Op}_{\Gcheck, x_0}^{\on{mer}})_{\on{Op}_{\Gcheck, x_0}^{\on{nilp-mon}}}$.

    We claim that the Verma modules
    $$
    \mathbb{M}^{\check{\lambda}} = \on{ind}_{I}^{(\hat{g}_{\on{crit}}, I)} k_{\check{\lambda}} \in \IKLx
    $$
    are sent to $\mathcal{O}_{\on{Op}_{\Gcheck, x_0}^{\check{\lambda}-\on{nilp}}}$ under the functor $\IKLx \to \IndCoh(\on{Op}_{\Gcheck, x_0}^{\on{mer}})_{\on{Op}_{\Gcheck, x_0}^{\on{nilp-mon}}}$.

    To see this, we will adapt the argument in \cite{FG} that shows $\on{FLE}$ sends the Weyl modules $\mathbb{V}^{\check{\lambda}}$ to $\mathcal{O}_{\on{Op}_{\Gcheck, x_0}^{\check{\lambda}-\on{reg}}}$:

    On one hand, we have the isomorphism
    $$
    \mathcal{O}_{\on{Op}_{\Gcheck, x_0}^{\check{\lambda}-\on{nilp}}} \simeq \on{End}(\mathbb{M}^{\check{\lambda}})
    $$
    and unwinding definitions yields
    $$
    \on{End}(\mathbb{M}^{\check{\lambda}}) = \on{Hom}(\on{ind}_{I}^{(\hat{g}_{\on{crit}}, I)} k_{\check{\lambda}}, \mathbb{M}^{\check{\lambda}}) \simeq \on{Hom}_{I}(k_{\check{\lambda}}, \mathbb{M}^{\check{\lambda}}) =: (\mathbb{M}^{\check{\lambda}})_{\check{\check{\lambda}}}^{I^0}.
    $$
    For any $M \in \gmod_{\on{crit}}$, there is a map
    $$
    (M)_{\check{\lambda}}^{I^0} \to M^{\hat{\mathfrak{n}}_{+}} \to H^{\frac{\infty}{2}}(\mathfrak{n}((t)), \mathfrak{n}[[t]], M \otimes \Psi_0),
    $$
    which is the content of Lemma 3 of \cite{FG}.
    Taking $M = \mathbb{M}^{\check{\lambda}}$ gives us a map
    $$
    \mathcal{O}_{\on{Op}_{\Gcheck, x_0}^{\check{\lambda}-\on{nilp}}} \to H^{\frac{\infty}{2}}(\mathfrak{n}((t)), \mathfrak{n}[[t]], \mathbb{M}^{\check{\lambda}} \otimes \Psi_0),
    $$
    where the latter identifies with the desired image of $\mathbb{M}^{\check{\lambda}}$. Thus, it suffices to show that this map is an isomorphism.

    Analogously to the $\check{\lambda}$-regular case, we have that the above map has the structure of a morphism of modules over the lie algebroid $N^*_{\on{Op}_{\Gcheck, x_0}^{\check{\lambda}-\on{nilp}}/\on{Op}_{\Gcheck, x_0}^{\on{mer}}}$, and that $\mathcal{O}_{\on{Op}_{\Gcheck, x_0}^{\check{\lambda}-\on{nilp}}}$ is irreducible as a $N^*_{\on{Op}_{\Gcheck, x_0}^{\check{\lambda}-\on{nilp}}/\on{Op}_{\Gcheck, x_0}^{\on{mer}}}$-module. Therefore, this map must be injective and it suffices to compute and compare the characters of both sides as in Section 5 of \cite{FG}.

    First, let us compute the character of $\mathcal{O}_{\on{Op}_{\Gcheck, x_0}^{\check{\lambda}-\on{nilp}}}$. Let 
    $$
    J = \{\alpha_i \in \Delta | \langle \alpha_i, \check{\lambda} + \check{\rho} \rangle = 0\}
    $$
    and denote by $\check{P}_J$ the parabolic subgroup of $\check{G}$ corresponding to $J$.
    Now note that $\on{Op}_{\Gcheck, x_0}^{\check{\lambda}-\on{nilp}}$ is the set of operators of the form
    $$
    \{ \partial_t + \Sigma_{\Delta} \ t^{\langle \alpha_i, \check{\lambda} \rangle} f_i + v(t) + \frac{w}{t} \ | \ v(t) \in \check{\mathfrak{b}}[[t]], w \in O\} 
    $$
    modulo the action of $\check{N}[[t]]$, where
    $$
    O = \{w \in \check{\mathfrak{b}} \ | \ w + \Sigma_{J}f_i \on{mod}\check{\mathfrak{n}}_J \in \mathcal{N}_{\check{\mathfrak{m}}_J}\}.
    $$
    However, this is the same as
    $$
    \{ \partial_t +\frac{1}{t}(\Sigma_{\Delta}f_i - (\check{\lambda} + \check{\rho})) + v(t) + \frac{w}{t} \ | \ v(t) \in (\check{\lambda} + \check{\rho})(t) \check{\mathfrak{b}}[[t]] (\check{\lambda} + \check{\rho})(t)^{-1}, w \in (\check{\lambda} + \check{\rho})(t) O (\check{\lambda} + \check{\rho})(t)^{-1} \}
    $$
    modulo the free action of $(\check{\lambda} + \check{\rho})(t) \check{N}[[t]] (\check{\lambda} + \check{\rho})(t)^{-1}$.

    Consider $O \cap \check{\mathfrak{m}}_J$, which carries a free and transitive action of the group $\check{N} \cap \check{M}_J$. There is a surjective homomorphism of groups
    $$
    \phi : (\check{\lambda} + \check{\rho})(t) \check{N}[[t]] (\check{\lambda} + \check{\rho})(t)^{-1} \to \check{N} \cap \check{M}_J
    $$
    which picks out the constant term, and there is a surjective morphism
    \begin{align}
        &\{ \partial_t +\frac{1}{t}(\Sigma_{\Delta}f_i - (\check{\lambda} + \check{\rho})) + v(t) + \frac{w}{t} \ | \ v(t) \in (\check{\lambda} + \check{\rho})(t) \check{\mathfrak{b}}[[t]] (\check{\lambda} + \check{\rho})(t)^{-1}, w \in (\check{\lambda} + \check{\rho})(t) O (\check{\lambda} + \check{\rho})(t)^{-1} \} \\
        &\overset{\pi}{\longrightarrow} O \cap \check{\mathfrak{m}}_J
    \end{align}
    which picks out the constant term of $w$. $\pi$ is equivariant with respect to $\phi$, so we have
    $$
    \on{Op}_{\Gcheck, x_0}^{\check{\lambda}-\on{nilp}} \simeq \pi^{-1}(0)/\on{ker}(\phi)
    $$
    and thus
    $$
    \on{char}(\mathcal{O}_{\on{Op}_{\Gcheck, x_0}^{\check{\lambda}-\on{nilp}}}) = \on{char}(\on{Fun}(\pi^{-1}(0)))/\on{char}(\on{Fun}(\on{ker}(\phi))).
    $$
    $\pi^{-1}(0)$ is
    $$
    \{ \partial_t +\frac{1}{t}(\Sigma_{\Delta}f_i - (\check{\lambda} + \check{\rho})) + v(t) + \frac{w}{t} \ | \ v(t) \in (\check{\lambda} + \check{\rho})(t) \check{\mathfrak{b}}[[t]] (\check{\lambda} + \check{\rho})(t)^{-1}, w \in (\check{\lambda} + \check{\rho})(t) \check{\mathfrak{n}}_J (\check{\lambda} + \check{\rho})(t)^{-1} \}
    $$
    so
    $$
    \on{char}(\on{Fun}(\pi^{-1}(0))) = \Pi_{n > 0} (1-q)^{-l} \cdot (\Pi_{\alpha \in \Phi_{J}^+} \Pi_{n \geq 0} (1 - q^{n+1})^{-1}) \cdot (\Pi_{\alpha \in \Phi^+ \setminus \Phi_{J}^+} \Pi_{n \geq 0} (1 - q^{n + \langle \alpha, \check{\lambda}+\check{\rho}\rangle})^{-1}).
    $$
    On the other hand,
    $$
    \on{ker}(\phi) = (\check{\lambda} + \check{\rho})(t) \check{N}[[t]] (\check{\lambda} + \check{\rho})(t)^{-1} \cap \on{exp}(t\check{\mathfrak{n}}[[t]])
    $$
    so 
    $$
    \on{char}(\on{Fun}(\on{ker}(\phi))) = (\Pi_{\alpha \in \Phi_{J}^+} \Pi_{n \geq 0} (1 - q^{n+1})^{-1}) \cdot (\Pi_{\alpha \in \Phi^+ \setminus \Phi_{J}^+} \Pi_{n \geq 0} (1 - q^{n + \langle \alpha, \check{\lambda}+\check{\rho}\rangle})^{-1}).
    $$
    Therefore,
    $$
    \on{char}(\mathcal{O}_{\on{Op}_{\Gcheck, x_0}^{\check{\lambda}-\on{nilp}}}) = \Pi_{n > 0} (1-q)^{-l}.
    $$

    Finally, to compute the character of $H^{\frac{\infty}{2}}(\mathfrak{n}((t)), \mathfrak{n}[[t]], \mathbb{M}^{\check{\lambda}} \otimes \Psi_0)$, we follow what is done in Section 5.2 of \cite{FG} essentially verbatim except with the finite Weyl module $V^{\lambda}$ replaced with the finite Verma module $M^{\lambda}$. Then, we find that 
    $$
    \on{char}(H^{\frac{\infty}{2}}(\mathfrak{n}((t)), \mathfrak{n}[[t]], \mathbb{M}^{\check{\lambda}} \otimes \Psi_0)) = \Pi_{n > 0} (1-q)^{-l}
    $$
    as desired.
\end{proof}

\begin{prop}
    $\IndCoh(\on{Op}_{\Gcheck, x_0}^{\on{mer}} \underset{\on{LS}_{\Gcheck, x_0}^{\on{mer}}}{\times} (\ncheck/\Bcheck)_{x_0})$ is Iwahori-Hecke tempered.
\end{prop}
\begin{proof}
For tautological reasons, there is an isomorphism
    $$
    \on{Op}_{\Gcheck, x_0}^{\on{mer}} \underset{\on{LS}_{\Gcheck, x_0}^{\on{mer}}}{\times} (\ncheck/\Bcheck)_{x_0} \simeq (\on{Op}_{\Gcheck, x_0}^{\on{mer}})_{\on{nilp-mon}}^{\wedge} \underset{(\on{LS}_{\Gcheck, x_0}^{\on{mer}})_{\on{nilp-mon}}^{\wedge}}{\times} (\ncheck/\Bcheck)_{x_0}.
    $$
    Because $(\on{LS}_{\Gcheck, x_0}^{\on{mer}})_{\on{nilp-mon}}^{\wedge}$ is passable like its spherical counterpart $(\on{LS}_{\Gcheck, x}^{\on{mer}})_{\on{mon-free}}^{\wedge}$,
    \begin{align}
    &\IndCoh((\on{Op}_{\Gcheck, x_0}^{\on{mer}})_{\on{nilp-mon}}^{\wedge} \underset{(\on{LS}_{\Gcheck, x_0}^{\on{mer}})_{\on{nilp-mon}}^{\wedge}}{\times} (\ncheck/\Bcheck)_{x_0}) \\ &\simeq \IndCoh((\on{Op}_{\Gcheck, x_0}^{\on{mer}})_{\on{nilp-mon}}^{\wedge}) \underset{\QCoh((\on{LS}_{\Gcheck, x_0}^{\on{mer}})_{\on{nilp-mon}}^{\wedge})}{\otimes} \QCoh(\ncheck/\Bcheck)_{x_0}
    \end{align}
    as $\on{Aff}_{\Gcheck, x_0}^{\on{spec}}$-module categories. However, the action of $\on{Aff}_{\Gcheck, x_0}^{\on{spec}}$ on the right hand side comes from the action of $\on{Aff}_{\Gcheck, x_0}^{\on{spec}}$ on $\QCoh(\ncheck/\Bcheck)_{x_0}$, which factors through $\on{Aff}_{\Gcheck, x_0}^{\on{spec, temp}}$.
\end{proof}

Finally, we will show that $\IKLx$ is Iwahori-Hecke tempered.

\begin{rem}
    Here the argument diverges significantly from the proof of the analogous result, Proposition 7.2.6 of \cite{GLC2} given in Section 7.4. This is because the action of $\on{Aff}_{\Gcheck, x_0}^{\on{spec}}$ on $\IKLx$ is not t-exact. We will instead use the following weaker property.
\end{rem}

\begin{lem}
    If
    $$
    \mathcal{F} \in \on{Aff}_{\Gcheck, x_0}^{\on{spec}, \leq -\infty}
    $$
    and 
    $$
    c \in \IKLx^c \subset \IKLx^b,
    $$
    then
    $$
    \mathcal{F} \star c \in \IKLx^{\leq -\infty}.
    $$
\end{lem}
\begin{proof}
    Since $\IKLx$ is generated by verma modules, we may take $c = \mathbb{M}^{\check{\lambda}}$.
    
    The functor
    $$
    \Gamma : \Dmod_{\on{crit}}(\on{Fl}_G) \longrightarrow \IKLx
    $$
    was studied in \cite{FG09}. For an object $\mathcal{F} \in \on{Aff}_{G}$, 
    $$
    \Gamma(\mathcal{F}) = \mathcal{F} \star \mathbb{M}^0.
    $$
    Using similar methods, one can study a twisted version
    $$
    \Gamma^{\check{\lambda}} : \Dmod_{\on{crit}, \check{\lambda}}(\on{Fl}_G) \longrightarrow \IKLx
    $$
    which sends
    $$
    \Gamma^{\check{\lambda}}(\mathcal{F}) = \mathcal{F} \star \mathbb{M}^{\check{\lambda}}.
    $$
    Now the claim follows from the bounded cohomological amplitude of each $\Gamma^{\check{\lambda}}$.
\end{proof}

\begin{prop}
    $\IKLx$ is Iwahori-Hecke tempered.
\end{prop}

\begin{proof}
    We want to show that
    $$
    \on{temp} : \IKLx \longrightarrow \IKLx^{\on{temp}}
    $$
    is an equivalence. Since it is a colocalization, it suffices to show that it is conservative.
    Note that the functor $\on{IFLE}^{\on{temp}}$
    $$
    \IKLx^{\on{temp}} \longrightarrow \IndCoh(\on{Op}_{\Gcheck, x_0}^{\on{mer}} \underset{\on{LS}_{\Gcheck, x_0}^{\on{mer}}}{\times} (\ncheck/\Bcheck)_{x_0})^{\on{temp}} \simeq \IndCoh(\on{Op}_{\Gcheck, x_0}^{\on{mer}} \underset{\on{LS}_{\Gcheck, x_0}^{\on{mer}}}{\times} (\ncheck/\Bcheck)_{x_0})
    $$
    is an equivalence. Thus, since $\on{IFLE}$ preserves compacts, so does $\on{temp}$. It follows that it suffices to show that $\on{temp}$ is conservative on compact objects. However, all compacts are bounded in $\IKLx$, and using the fact that
    $$
    \on{cone}(\on{temp}(\boldsymbol{1}_{\on{Aff}_{\Gcheck, x_0}^{\on{spec}}}) \to \boldsymbol{1}_{\on{Aff}_{\Gcheck, x_0}^{\on{spec}}}) \in \on{Aff}_{\Gcheck, x_0}^{\on{spec}, \leq -\infty}
    $$
    we can conclude that the cone of the counit map
    $$
    \on{cone}((\on{temp}(c) \to c) \simeq \on{cone}(\on{temp}(\boldsymbol{1}_{\on{Aff}_{\Gcheck, x_0}^{\on{spec}}}) \to \boldsymbol{1}_{\on{Aff}_{\Gcheck, x_0}^{\on{spec}}}) \star c \in \IKLx^{\leq - \infty}.
    $$
    Therefore, if $\on{temp}(c) = 0$, then $c \in \IKLx^{\leq - \infty}$ which implies $c = 0$ since c is also bounded. This shows that $\on{temp}$ is conservative on compact objects, as desired.
\end{proof}

%%%%%%%%%%%%%%%%%%%%%%%%%%%%%%%%%%%%%%%%%%%%%%%%%%%%%%%%%%%%%%%%%%%%%%%%%%

\end{document}